\documentclass[twoside]{article}

\usepackage{PRIMEarxiv}		% ARXIV

\usepackage[utf8]{inputenc} 	% allow utf-8 input
\usepackage[T1]{fontenc}    	% use 8-bit T1 fonts
\usepackage{url}            		% simple URL typesetting
\usepackage{booktabs}       	% professional-quality tables
\usepackage{nicefrac}       	% compact symbols for 1/2, etc.
\usepackage{microtype}      	% microtypography
\usepackage{fancyhdr}       	% header
\usepackage{graphicx, xcolor} 	% graphics
\usepackage[caption=false]{subfig}
\usepackage{latexsym, amsthm, amsfonts, amsmath, amscd, eucal, amssymb, epsf, mathtools, bbm}		
\usepackage{hyperref, cleveref}	% to manage cross references and hyperlinks
\usepackage{comment}

\usepackage{tikz}
\usetikzlibrary{patterns}
\usepackage{pgfplots}

\newtheorem{theorem}{Theorem}

\newtheorem{claim}[theorem]{Claim}

\newtheorem{definition}[theorem]{Definition}
\newtheorem{lemma}[theorem]{Lemma}

\newtheorem{remark}[theorem]{Remark}

\providecommand{\Keywords}[1]{\textbf{\textit{Key words:\ }} #1}
\providecommand{\MSCcodes}[1]{\textbf{\textit{Mathematics Subject Classification:\ }} #1}

\usepackage[normalem]{ulem}

\title{Non-negativity and zero isolation for generalized mixtures of densities
\thanks{%thanks 
{Fundings: part of this research resulted from the SFI projects 07/MI/008 and RFP2007-MATF802, and the European Union - FSE REACT-EU, PON Ricerca e Innovazione 2014-2020} }
}

\author{
  Stefano Bonaccorsi 
  \thanks{Department of Mathematics, Universit\`a degli Studi di Trento,  Italy, (\texttt{stefano.bonaccorsi@unitn.it}, \texttt{giulia.lombardi@unitn.it}) }
   \And
  Bernard Hanzon 
  \thanks{School of Mathematical Sciences, University College Cork, Ireland, \texttt{b.hanzon@ucc.ie} }
\And
  Giulia Lombardi  \footnotemark[2]
}

\begin{document}
\maketitle

%% ------------------------------------------------------------------
%% ABSTRACT
%% ------------------------------------------------------------------
%\begin{tcbverbatimwrite}{tmp_\jobname_abstract.tex}
\begin{abstract}
In the literature, finite mixture models are described as linear combinations of probability distribution functions having the form
$\displaystyle f(x) = \Lambda \sum_{i=1}^n w_i f_i(x)$,  $x \in \mathbb{R}$,
where $w_i$ are positive weights, $\Lambda$ is a suitable normalising constant and $f_i(x)$ are given probability density functions.
The fact that $f(x)$ is a probability density function
follows naturally in this setting.
Our question is: 
\emph{what happens when we remove the sign condition on the coefficients $w_i$?}
\\
The answer is that it is possible to determine the sign pattern of the function $f(x)$ by an algorithm based on finite sequence that we call a \emph{generalized Budan-Fourier sequence}. In this paper we provide theoretical motivation for the functioning of the algorithm, and we describe with various examples its strength and possible applications.
\end{abstract}


\MSCcodes{
%\subjclass[2020]{
60E05, 62H30
}
%\end{tcbverbatimwrite}
%\input{tmp_\jobname_abstract.tex}
%% ------------------------------------------------------------------
%% END HEADER
%% ------------------------------------------------------------------

\section{Introduction}\label{sec:intro}

The first instance of a finite mixture probabilistic model goes back to the work of Pearson \cite{Pearson1894}, where a two-component normal mixture is used to fit a dataset of body lengths of crabs from the bay of Naples.
Since then, they have been used in the (statistical) description of a variety of problems: for a survey of these applications, we refer for instance to \cite{Titterington1985, McLachlan2000, Fruhwirth2006}.
The analysis of neural networks having response functions of this form, without limitation on the sign of the coefficients, goes back at least to the paper \cite{Park1991} (where finite Gaussian mixtures are studied).
It is worth mentioning the recent applications to cluster analysis \cite{Everitt2011} and unsupervised Machine Learning problems %, see {\it e.g.}\/ Aquilanti \emph{et al.} 
\cite{Aquilanti2020, Aquilanti2021}.
Although in this paper we shall only consider univariate distributions, the applications of finite mixture models naturally extend to the multivariate case.

We shall further consider a broader definition of finite mixture models rather than Gaussian mixtures, taking into account functions of the form
\begin{align}
\label{e1}
f(x) = \Re \left( \sum_{i=1}^n p_i(x) e^{q_i(x)} \right), \qquad x \in I,
\end{align}
where $\Re(z)$ denotes the real part of the complex number $z$, and $\{p_i(x),\ i=1, \dots, n\}$ and $\{q_i(x),\ i=1, \dots, n\}$ are given complex-valued polynomials
on an interval $I \subset \mathbbm{R}$ which we allow to be finite, or semi-infinite (such as $I = \mathbbm{R}_+ \colon= \{x \ge 0\}$), or infinite ($I = \mathbbm{R}$).

For example, consider a mixture of Erlang distributions
$f_i(x) = x^{m_i - 1} e^{-\lambda_i x}$, $x \in [0,\infty)$, which leads to the  distribution 
\begin{align}\label{eq:ERL-example}
f(x) = \Lambda \sum_{i=1}^n w_i x^{m_i-1} e^{-\lambda_i x}, \qquad x \in [0,\infty),\ \lambda_i > 0, i=1, \dots, n
\end{align}
for a suitable normalizing constant $\Lambda$, where 
either $\lambda_{i+1} > \lambda_i$, or $\lambda_{i+1} = \lambda_i$ and $m_{i+1} > m_i$.
If the coefficients $w_i$ are positive, then
$f$ is a probability density function. 
%Our question is: 
%\emph{what happens when we remove the sign condition on the coefficients of $w_i$?}
%
%\vskip 1\baselineskip
%

Suppose that the Erlang mixture in \eqref{eq:ERL-example} has the further property that the scale parameters are the same for all the terms: $\lambda_i = \lambda$.
Then, the problem of proving that $f$ is a probability density function reduces to the study of the positivity of the polynomial $\displaystyle p(x) = \sum_{i=1}^n w_i x^{m_i-1}$. 
This is a classical problem in the mathematical literature and can be tackled in several ways.
We focus here on the \emph{Budan-Fourier} approach, 
introduced in \cite{Gonzalez1998}, see also \cite{Coste2005, Galligo2013}.
The idea underlying this approach is that there
exists a stepwise constant, non-increasing function $V_p(x)$, taking integer values, with $V_p(-\infty) = \kappa \colon= {\rm deg}(p)$ and $V_p(\infty) = 0$, such that the jump points of $V_p(x)$ correspond to the \emph{virtual roots} (which include the real roots of $p$%and any multiple root of its derivatives
) of $p$. 
The \emph{Budan-Fourier theorem} states that the number of real roots of $p$ in the interval $(a,b]$ (counted with multiplicity) is at most $V_p(b) - V_p(a)$, and the defect is an even integer.

The Budan-Fourier approach is based on a sequence of functions associated with the derivatives of the function that is to be investigated. 
In the original BF approach this sequence is the sequence of higher order derivatives; in \cite{Coste2005}, $\pmb f$-derivatives are used, which allows to extend the range of application of the result.
In general, however, if a function is not a polynomial, the sequence of nontrivial higher order derivatives is not finite and the method cannot be applied. 
In our approach the sequence of higher order derivatives is replaced by another finite sequence that we will call a \emph{generalized Budan-Fourier sequence} (GBF sequence).

In   \cref{s:GBF} we present in detail  the functioning of the GBF Algorithm. 
This method has been formerly proposed in \cite{Hanzon2012} with the purpose of computing all the sign-changing roots of EPT (exponential-polynomial-trigonometric) functions within a finite interval.
We provide an application to EPT functions in  \cref{s:EPT}.
Next,  \cref{s:GGM} and \cref{sec:ggm} are devoted to the analysis of Gaussian mixtures. 
There we first construct the GBF sequence in the general framework of Gaussian mixtures with polynomial coefficients
%\begin{align*}
$\displaystyle f(x) = \sum_{j=1}^n p_j(x) \, e^{q_j(x)}$,
%\end{align*}
where $\{q_j(x)\}$ are second order polynomials with real coefficients and negative leading term.
We provide a simple, yet completely explicit example, of a polynomial Gaussian mixture depending on a parameter, and we show that the behaviour of the mixture may be very different according to the values of this parameter.

Next \cref{sec:ggm} is  devoted to finite Gaussian mixtures  $f(x) = \sum_{j=1}^n \gamma_j \, e^{q_j(x)}$.
These particular densities have become
very popular in several fields (such as speech recognition or image analysis \cite{Yu2015, Santosh2013}) basically because, in many cases, data are multimodal, i.e., the underlying population is already a combination of different sub-populations.
In such applications, it makes intuitive sense to model multimodal data as a mixture of unimodal Gaussian distributions, in order to approximate the original distribution at best.
Furthermore, modelling with finite Gaussian mixtures  maintains many of the theoretical and computational benefits of Gaussian models, making them practical for efficiently modeling very large datasets.
As an example, we recall that the sum of a random variable with Gaussian mixture distribution and an independent Gaussian random variable is distributed again as a Gaussian mixture, with a natural shift in the parameters.
In addition to the properties mentioned above, there is an approximation theorem \cite{Park1991} which proves that finite Gaussian mixtures have the capability to approximate ``arbitrary'' probability distribution functions rather well at least for large enough $N$, where $N$ is the size of the mixture, and a good spread of means and variances, once one allows for coefficients of arbitrary sign.
Summing up, although in the literature it is very common to find finite Gaussian mixtures represented as a linear combination of Gaussian probability density functions with positive weights, where this assumption ensures that the resulting Gaussian mixture does not have any sign-changing zero along the $x$-axes, to take real advantage of this class of distributions it seems necessary to enlarge the study to cover the study of mixtures with coefficients of arbitrary sign.

%In that section we also present the numerical results of a software, designed to work in Matlab, that allows the investigation of the zeros of a finite Gaussian mixture.
%In particular, given a finite Gaussian mixture (through the introduction of the means, variances and mixing coefficients) the
%software returns
%a maximal interval where the sign-changing roots may appear, an interval that is always bounded; further, if these roots exist, the software identifies them within a desired accuracy level. 
%As a result, the algorithm will allow us to identify those cases in which the mixture is non-negative on the real line.

In that section we also present some numerical results about the investigation of the zeros of a finite Gaussian mixture.
In particular, given a finite Gaussian mixture (through the introduction of the means, variances and mixing coefficients) 
we get
a maximal interval where the sign-changing roots may appear, an interval that is always bounded; further, if these roots exist, we may identify them within a desired accuracy level. 
As a result, we are able to identify those cases in which the mixture is non-negative on the real line.

Finally,  \cref{s:Was} presents an application of our results to the study of Wasserstein-1 distance between probability mixtures; 
in particular, we show that the knowledge of the zeros of the difference between the probability distribution functions, that is obtained  by using GBF sequences, is a necessary tool in computing the distance.
In accordance with our construction, we discuss in separate subsections the case of  EPT functions and Gaussian mixtures.

\section{Presentation of the basic building blocks of the approach}
\label{s:GBF}

For any sufficiently often differentiable function $f: \mathbb{R} \to \mathbb{R}$ we denote by $Df(x)$ the first order derivative in $x$, and generically $D$ the derivative operator;
similarly, $D^jf(x)$ is the $j$-th order derivative, for $j \in \mathbb{N}$ and $D^j$ is the correspondent operator.
\\
For any (real or complex valued) square matrix $A$, $|A|$ is the determinant of $A$.

\subsection{Generalized Budan-Fourier sequence}

%In this section we analyse in detail the functioning of the GBF Algorithm. This method has been formerly proposed in \cite{Hanzon2012} with the purpose of computing all the sign-changing roots of EPT (exponential-polynomial-trigonometric) functions within a finite interval.

Let $I$ 
%\textcolor{red}{$I = [a,b]$ a closed interval in the real line, where $a=-\infty$ and/or $b=+\infty$ are allowed,} 
be a given interval and $f \not\equiv 0$ a  function on $I$. 
Suppose that %In particular, 
$f$ is 
%real meromorphic (i.e. the quotient of two real analytic functions) 
%piecewise 
real analytic 
%\textcolor{red}{on $I = [a,b]$}
and hence 
the set of zeros of $f$ is discrete and has finite cardinality on any %!
closed and 
bounded sub-interval $[a,b]$ of $I$. 
We are interested in the \emph{sign changing zeros} of $f$, i.e., points $x_0$ such that $f(x_0) = 0$ and for which there exists an open neighbourhood $(x_0-\epsilon,x_0+\epsilon) \subset I, \epsilon>0$ such that for all $y,z$ in this neighbourhood with $y < x_0 < z$ it holds that $f(y) \cdot f(z) < 0$. 
Here we will present a method to determine all sign changing zeros on a given closed and bounded interval $[a,b]$. 
We shall denote by $R(f,[a,b])$ the set of sign-changing zeros of $f$ on the interval $[a,b]$.
Note that for any given point, so also for the points $a$ and for $b$, we can determine whether it is a sign-changing zero of $f$ by inspection. 
Therefore we can focus on determining the sign-changing zero in the open interval $(a,b)$.
An open interval is called a \emph{simple interval} for a function $f$ if it contains at most one sign-changing root of $f$ (possibly of multiplicity greater than 1).
A \emph{simple grid} for $f$ on the interval $[a,b]$ is a finite sequence of points $G = \{x_1 < \dots < x_n\} \subset I$ such that $(a,x_1),(x_n,b)$ and  $(x_i, x_{i+1}), i=1,2,\ldots,n-1$ are all simple intervals for $f$.
%We shall denote by $\Gamma(f,[a,b])$ the set of all simple grids for a function $f$ on the interval $[a,b].$ 
%\textcolor{red}{NOTE: for the first lines, we use generally $I$, now it is $[a,b]$... make uniform}

\begin{remark}
%Assume that $f$ is a smooth function on a bounded interval $I = [a,b]$. Then 
The set of sign-changing zeros for the first derivative, $R(f',[a,b])$, is
a simple grid for the function $f$.
\end{remark}

\begin{definition}\label{GBFdef}
Let $f$ be a 
%smooth 
real analytic function on the interval $[a,b]$. %a bounded interval $I$.
A sequence of functions $\displaystyle \{\psi_i,\ i=0, \dots, n\}$ is called a \emph{Generalized Budan-Fourier (GBF) sequence associated to $f$} 
and a sequence of functions $\displaystyle \{\varrho_i,\ i=1, \dots, n\}$ is called the \emph{associated sequence of pivots} if the following properties are satisfied:
\begin{enumerate}
\item $\psi_0 = f$;
\item $\psi_i$ and $\varrho_i$ are quotients of real analytic functions on $[a,b]$ for each $i=1,\ldots,n;$
\item $\psi_n$ is the zero function;
%there exists a $k \in \{0, 1, \dots, n\}$ such that $\psi_k(x)$ has no roots;
\item For each $k=1,2,\ldots,n,$ $D \left(\frac{\psi_{k-1}}{\varrho_k}\right)=\frac{\psi_k}{\varrho_k}$ for every value of $x$ for which both sides are defined.
\end{enumerate}
\end{definition}

In the case $f(x)$ is a polynomial of degree $n-1$, then the Budan-Fourier sequence is given by setting, for any $i = 1, \dots, n$, $\varrho_i=1$ and $\psi_i = f^{(i)}$.
Then, as noted above, a simple grid for $\psi_{i-1}$ is given by the set of sign-changing zeros of $\psi_i$. Moreover, the $n$-th derivative of $f(x)$ is the zero function.
However, this same procedure does not work even for simple analytical functions, such as, for instance, $f(x) = \sin(x)+x \cos(x)$.

\begin{remark} 
In the application of this definition the pivots are functions for which the real zeros and poles can be determined independently, such as is the case for instance for  polynomials. 
\end{remark}

\begin{remark}
We shall use repeatedly in the sequel the following observation. Let $\mathcal{G} = \{x_1,\dots,x_n\}$ be a simple grid for a function $f$ on the interval $[a,b]$, given by the sign-changing roots of a function $h$. Then the same grid is a simple grid for the function $\lambda \cdot f$, where $\lambda$ is a continuous, non vanishing function on $(a,b) \setminus \mathcal{G}$.
\end{remark}

\begin{remark}\label{rem:BH5}
Note that if a simple interval $(c,d)$ for a function $f$ is given, then by inspecting the signs of $f$ at $c$ and $d$ (actually the signs of $f(c+\epsilon)$ and $f(d-\epsilon)$ for sufficiently small positive $\epsilon$ to be more precise) one can determine whether $f$ has a sign-changing zero in the open interval and if so one can determine the sign-changing zero with arbitrary precision using a bisection algorithm.
\end{remark}

\begin{remark}\label{rem:BH6}  

Suppose the pivot functions are such that we can independently determine their zeros and poles. 
Then from the definition:
$$D\left(\frac{\psi_{n-1}}{\varrho_n}\right)=\frac{\psi_n}{\varrho_n}$$
and since by assumption $\psi_n \equiv 0$, it follows that $\psi_{n-1}$ is a multiple of $\varrho_n$ and hence $\psi_{n-1}$ inherits the poles and zeros of $\varrho_n$. 
\\
For any $k$, once one knows the poles as well as the sign-changing zeros of $\psi_k$ one can form a grid containing those  poles and sign-changing zeros as well as all the zeros and poles of $\varrho_k$. 
It then follows that the grid is a simple grid for $\psi_{k-1}$. 
The reason is that for any interval in between two consecutive grid points both the sign of $\varrho_k$ and that of $\psi_k$ does not change and the functions $\frac{\psi_{k-1}}{\varrho_k}$ and $\frac{\psi_{k}}{\varrho_k}$ are well-defined as they have no poles. 
Therefore  $\frac{\psi_{k-1}}{\varrho_k}$ is monotonic on such an interval and hence $\psi_{k-1}$ has at most one sign-changing zero (and no poles) on such an interval.

%Let $f$ be real analytic on $I$ and let $\{\phi_i\}_{i=1}^n$ be a GBF sequence for $f$ on $I$. 
%Then $\phi_n$ has no sign-changing roots and, by definition, we have that $I$ constitutes a simple interval for $\phi_{n-1}$. 
Therefore, using the previous remark one can determine all the  sign-changing zeros of $\psi_{n-1}$ (if they exist) using a bisection procedure and by repeating this procedure for $k= n-2,n-3,\ldots,1$, one can determine all the sign-changing zeros of $f$. 
\end{remark}

Mathematically there is one issue with the procedure outlined in this last remark: we have to assume that the grid points can be computed sufficiently precise to be able to determine the sign of the relevant function at each of those points. 
In case the true signs are either positive or negative, this assumption is correct if the calculations are done with sufficiently good precision. 
However if the sign to be determined is zero, then this assumption is, mathematically speaking, not correct in general. 
For polynomials with rational coefficients one can employ the Euclidean algorithm to determine if two such polynomials have a common zero and this can be used to resolve this issue in the case the $\psi_i$ are polynomials. 
For more general functions only partial results are know to the best of our knowledge. 
See for instance \cite{Richardson1997}.
In our experience so far the issue has not led to problems in practical calculations.
%\end{remark}

%\begin{remark}\label{rem:BH7}
%The functions in the GBF sequence that we will encounter can be real analytic or real meromorphic (i.e. the quotient of two real analytic functions). The sign pattern of a meromorphic function outside of the (isolated) pole locations is the same as that of the real analytic function that is obtained by multiplying the meromorphic function by the square of the denominator. In this way we can replace any meromorphic function in the GBF sequence by a real analytic one, without changing the sign pattern, outside of the (isolated) pole locations.
%\end{remark}

\subsection{Pivot functions and the construction of a GBF sequence}

In this subsection we will describe one strategy to arrive at a sequence of pivots and hence at a GBF sequence for a given function, without dealing with the details at this point. These will be dealt with for a number of classes of functions in the following sections. 
Suppose  that $f$ is a function in the linear span $V$ of $n$ linearly independent real analytic functions $\{h_1, \dots, h_n\}.$ 
Let us assume (and in the sequel we shall discuss when this assumption actually holds) that there exists an $n-$th order linear differential operator $\Phi$ such that
\begin{align*}
    \Phi(h_j) = 0 \qquad j=1, \dots, n.
\end{align*}
Then, by linearity, it holds that $\Phi(f) = 0$ as well.

Assume we can factorize $\Phi$ into first order linear differential operators,
\begin{align*}
    \Phi = \Phi_n \circ \dots \circ \Phi_1,
\end{align*}
where for any $j$ we have $\Phi_j(f) = D f + b_j f$ for some function $b_j$, for $f \in V$.
Then we can associate with this a growing sequence of subspaces of $V$, namely
\begin{align*}
    V_1 &= \mathop{\rm Ker}(\Phi_1),
    \\
    V_2 &= \mathop{\rm Ker}(\Phi_2 \circ \Phi_1),
    \\
    \dots &
    \\
    V_n &= \mathop{\rm Ker}(\Phi_n \circ \dots \circ \Phi_1) = \mathop{\rm Ker}(\Phi).
\end{align*}

\begin{claim}
Under certain regularity conditions there will exist a basis $\{\tilde h_1, \dots, \tilde h_n\}$ of $V = \mathop{\rm span}(h_1, \dots, h_n)$ such that
\begin{align*}
    V_1 &= \mathop{\rm span}(\tilde h_1),
    \qquad
    V_2 = \mathop{\rm span}(\tilde h_1, \tilde h_2),
    \qquad
    \dots 
    \qquad
    V_n = V = \mathop{\rm span}(\tilde h_1, \dots, \tilde h_n).
\end{align*}
\end{claim}

Notice that it holds
\begin{align*}
    \Phi_j \circ \dots \circ \Phi_1(\tilde h_k) = 0
\end{align*}
provided $j \ge k$.
\\
Now we set
\begin{align*}
    \varrho_1 &= \tilde h_1,
    \qquad
    \varrho_2 = \Phi_1(\tilde h_2),
    \qquad
    \dots 
    \qquad
    \varrho_n = \Phi_{n-1} \circ \dots \circ \Phi_1(\tilde h_n).
\end{align*}
We claim under certain regularity conditions  $\{\varrho_1, \dots, \varrho_n\}$ can play the role of the sequence of \emph{pivot} functions in  \cref{GBFdef}.

%\begin{claim}
%If we know the sign-changing zeros of the pivots, then the sequence %$\{p_1, \dots, p_n\}$ can be used as a  sequence of pivot functions.
%\end{claim}

Now, since $f \in V$, there exists a representation of $f$ in terms of the basis $\{\tilde h_1, \dots, \tilde h_n\}$
\begin{align*}
    f = \sum_{j=1}^n f_j \tilde h_j, ~f_j \in \mathbb{R},~j=1,2,\ldots,n.
\end{align*}
Then 
\begin{equation*}
\begin{aligned}
    \Phi_1(f) =& \sum_{j=2}^n f_j \Phi_1( \tilde h_j), \\ %\qquad
    \Phi_2 \circ \Phi_1(f) =& \sum_{j=3}^n f_j \Phi_2 \circ \Phi_1( \tilde h_j),\\ 
    \vdots& \\
    \Phi_{n-2} \circ ...\Phi_2 \circ \Phi_1(f) =& \sum_{j=n-1}^n f_j \Phi_{n-2} \circ...\circ \Phi_1( \tilde h_j) \\%\qquad
    \Phi_{n-1} \circ ...\circ \Phi_1(f)=&f_n \Phi_{n-1} \circ ...\circ \Phi_1(f) \\
\\
\Phi_{n} \circ ...\circ \Phi_1(f)=&0.
  \end{aligned}
\end{equation*}
%Recall that $D$ denote the differential operator on $\mathbbm{R}$; with no loss of generality we can assume that $\Phi_j(f) = D f + b_j f$ for some function $b_j$.
We notice that 
\begin{equation}
    \Phi_k(f) = \varrho_k \, D \left( \frac{f}{\varrho_k}\right), \quad k = 1, \dots, n
\end{equation}
which implies, in particular,
\begin{equation}
    \Phi_k \circ \dots \circ \Phi_1(f) = \varrho_k D \left( \frac{\varrho_{k-1}}{\varrho_k} D \left( \frac{\varrho_{k-2}}{\varrho_{k-1}} \dots \left(\frac{\varrho_1}{\varrho_2} D \left(\frac{f}{\varrho_1} \right) \right) \right) \right), ~k=1,2,\ldots,n
\end{equation}
and that, if $\varrho_1, \varrho_2,\ldots, \varrho_n$ can indeed act as the sequence of pivots, the corresponding GBF sequence for $f$ is given by 
\begin{equation}
    \psi_k=\Phi_k(f), ~k=1,\ldots,n.
\end{equation}

In next subsection, we shall use the \emph{Polya-Ristroph formula} to obtain such a  sequence of operators $\{\Phi_j,\ j = 1, \dots, n\}$ explicitly and show that we can actually do that in such a way that $h_j=\tilde{h}_j, ~j=1,\ldots, n.$
%Actually, the above construction is independent of the way of getting a factorization of $\Phi$ into linear first order differential operators; 
%whatever way we use to obtain the pivot functions, if we are able to study the sign-changing zeros of the sequence $\{p_k,\ k = 1, \dots, n\}$ we may obtain a GBF sequence in this way.

\subsection{A formula by Polya and Ristroph}

We will make use of a classical formula of G.\ Polya  \cite{Polya1922} and results from R.\ Ristroph \cite{Ristroph1972} which state the following.

Let $\{h_i\}$ be a finite sequence of real, linearly independent real analytic functions on a given open interval $I$. Let, for $m \in \mathbbm{N}$,
\begin{align*}
W_m(x) = W(h_1, \dots, h_m)(x) = \begin{vmatrix}
h_1 & h_2 & \dots & h_m \\ Dh_1 & Dh_2 & \dots & Dh_m \\ \dots & & & \dots
\\ D^{m-1}h_1 & D^{m-1}h_2 & \dots & D^{m-1}h_m \end{vmatrix}(x)
\end{align*}
be the Wronskian defined on $I$. We set $W_0(x) = 1$ and we notice that $W_1(h_1) = h_1$ and $W_2(h_1,h_2) = h_1 (Dh_2)  - (Dh_1) h_2$. 
Denote, for any real analytic function $f$ %, defined and sufficiently smooth on $I$,
\begin{align}\label{eq:Km1}
K_m(f) = \frac{W(h_1, \dots, h_m, f)}{W(h_1, \dots, h_m)}.
\end{align}
Then $f \mapsto K_m(f)$ is the {\em unique} monic\footnote{i.e., with leading coefficient equal to 1}\ $m$-th order linear differential operator 
for which $\{h_i,\ i=1, \dots m\}$ is a fundamental set.

\smallskip

\begin{remark}
We introduce the following expansion for the differential operator $K_m$
\begin{align}\label{eq:Km1-1}
K_m(f) = \sum_{j=0}^m a_{m,j} D^j f, \qquad a_{m,m}=1
\end{align}
%where $D$ again denotes the derivative with respect to the real variable $x$ 
(the dependence on $x$ is suppressed for notational simplicity).

We remark that for $0 \le j \le m-1$,  $a_{m,j}$ is 
the quotient of two real analytic functions, and its poles are therefore  zeros of the denominator and hence of $W_m(x)$.
Since further
(cf.\ Polya's paper \cite{Polya1922})
\begin{align*}
D W_m(x) %= D  \begin{vmatrix} h_1 & h_2 & \dots & h_m \\ h'_1 & h'_2 & \dots & h'_m \\ \dots & & & \dots
%\\ h^{(m-1)}_1 & h^{(m-1)}_2 & \dots & h^{(m-1)}_m \end{vmatrix} 
= - a_{m,m-1}(x) W_m(x),
\end{align*}
it follows that whenever $W_m(x)$ does not vanish inside $I$, then the representation
\begin{align}\label{eq:Wm-3}
W_m(x) = \exp\left( -\int_{x_0}^x a_{m,m-1}(z) \, {\rm d}z \right) W_m(x_0)
\end{align}
holds for any $x, x_0 \in I$.
\end{remark}

\smallskip

\begin{remark}
The following expression of the
differential operator $K_m(f)$ is proved in \cite{Polya1922} by a direct computation involving the derivative of the Wronskian, and in \cite{Ristroph1972} via a direct argument based on induction:
\begin{equation}\label{eq:Km2-1}
K_m(f) 
= { \tfrac{W_m}{W_{m-1}} D \left( \tfrac{W^2_{m-1}}{W_m W_{m-2}} D \left( \tfrac{W^2_{m-2}}{W_{m-1} W_{m-3}} \dots D \left( \tfrac{W_2^2}{W_3 W_1}
D \left( \tfrac{W_1^2}{W_2} D \left(\tfrac{f}{W_1} \right) \right) \dots \right) \right. \right.}
\end{equation}
In particular, for $m=2$ we obtain
\begin{equation}\label{eq:Km2-2}
K_2(f) 
= \tfrac{W_2}{W_1}
D \left( \tfrac{W_1^2}{W_2} D \left(\tfrac{f}{W_1} \right) \right).
\end{equation}
Note that \eqref{eq:Km2-1} and \eqref{eq:Km2-2} hold for all $x$ for which none of the of the denominators involved are zero.
\end{remark}
From equation \eqref{eq:Km2-1} it follows that in case the Wronskians $W_j,~j=1,\ldots,n$ are functions for which the zeros can be obtained independently (e.g. if they are the (real) zeros of polynomials) then the quotients $\varrho_k := \frac{W_k}{W_{k-1}},~k=1,2,\ldots,n,$ where $W_0\equiv 1$ by convention, can play the role of the sequence of pivots and then $K_m(f),~m=1,2,\ldots,n$ is the associated GBF sequence in case $f$ lies in the linear span of $\{h_1,\ldots,h_n\}.$

In the following four sections we will use the general framework presented here to develop pivot sequences and corresponding GBF sequences for a number of classes of functions and explain how these can be used to obtain all the sign-changing zeros of a given function on a given closed interval $[a,b].$ It will also be explained how for some cases it is possible to extend the results to obtain all sign-changing zeros on the real line.

\section{Analysis of the class of Exponential-Polynomial-Trigonometric functions}
\label{s:EPT}

Consider a function $f:[0,\infty) \to \mathbbm{R}$ of the form \eqref{e1}. 
We say that $f$ is a real \emph{EPT (Exponential-Polynomial-Trigonometric) function} if
\begin{align}\label{e:ept}
f(x) = \Re \left( \sum_{k=1}^m p_k(x) e^{\mu_k x} \right), \qquad x \ge 0,
\end{align}
where $m \ge 1$, $\mu_k \in \mathbbm{C}$ and $p_k(x)$ are complex-valued polynomials, for $k=1,\dots, m$.
We shall denote by $\mathcal{B}$  the set of real EPT  functions.
Let us state a few alternative characterizations, that may be useful in the sequel, see for instance \cite{Sexton2013}.

\begin{lemma}
$f \in \mathcal{B}$ if and only if any of the following equivalent characterizations hold:
\begin{enumerate}
    \item there exist a real $\tilde{n} \times \tilde{n}$ matrix $A$, a real $\tilde{n}$-dimensional column vector $b$ and a real $\tilde{n}$-dimensional row vector $c$ such that $f(x) = c e^{A x}b$. 
    \\
    If $n$ is the minimal possible choice for $\tilde{n},$ given $f \in \mathcal{B}$, then we say that $(A,b,c)$ is a minimal realization of $f$ and $n$ is the \emph{order} (or \emph{McMillan degree}) of the function $f$.
    \item the Laplace transform of $f$ is a strictly proper rational function.
    \item $f$ is the solution of a linear differential equation with constant coefficients (Euler-D'Alembert class).
\end{enumerate}
\end{lemma}

This last result amounts to saying that there exists a linear differential operator $p(D)$ such that $p(D)f = 0$. 
The associated polynomial $p(x)$ 
has degree ${\mathop {\rm deg}}(p) \ge n$
and if we search for such an operator $p(D)$ having minimal degree we obtain ${\mathop {\rm deg}}(p) = n$ ($n$ is the \emph{McMillan degree}).
If we search for a minimal degree, monic polynomial we have uniqueness of $p(D)$ and the representation
\begin{align*}
p(x) = %{\mathop {\rm det}} 
\left| xI - A \right|,
\end{align*}
where $A$ is the matrix from a minimal realization $(A,b,c)$ of $f$.

Since
$p(x)$ is a real polynomial of degree $n$, it can be factorized into linear and quadratic real factors
\begin{align}\label{eq:factor-p}
p(x) = p_1(x) p_2(x) \dots p_{a+b}(x),
\end{align}
where $p_1, \dots, p_a$ are linear factors $p_i = x-\lambda_i,$ and $p_{a+1}, \dots, p_{a+b}$ are irreducible quadratic factors
$p_i(x) = x^2 - 2 \theta_i + \rho_i^2$ (with $|\theta_i| < \rho_i$ \footnotemark), and $a + 2b = n$.
\footnotetext{The roots of $p_i$ are $\theta_i \pm i \sqrt{\rho_i^2 - \theta_i^2}$, which can be written as
$\rho_i e^{\pm i \arccos(\theta_i/\rho_i)}$}

\subsection{GBF sequence for EPT functions}

Now, we construct a GBF sequence associated
to an EPT function with minimal realization $f(x) = c e^{A x}b$. We know that $p(D)f(x) = 0$ for
$p(x) = \left| xI - A \right|$.
We start by setting $\psi_0 = f$.

Let $p_1(x)$ be the first factor of $p(x)$, as constructed in \eqref{eq:factor-p}, and assume that $p_1(x) = x - \lambda_1$ is a linear factor.
We set $h_1(x) = e^{\lambda_1 x}$, so that $p_1(D)h_1 = 0$.
Define
\begin{align*}
\psi_1(x) = p_1(D) \psi_0(x) = h_1 D \left( \frac{\psi_0}{h_1} \right).
\end{align*}
Then, since $h_1(x) > 0$ definitely, it follows that on any interval on which $\psi_1$ does not change sign, then $h_1^{-1} \psi_0$ is {\em monotonic},
hence the interval will be a {\em simple interval} for $\psi_0$. 
Therefore, the set $R(\psi_1,[a,b])$ of sign-changing zeros for $\psi_1$ defines a simple grid for $\psi_0$.
We proceed recursively for any $i=1, \dots, a$ ($a$ being the number of linear factors in $p(x)$), thus constructing a sequence
$\{\psi_0, \psi_1,\dots, \psi_a\}$, where $\psi_{i} = p_{i}(D) \psi_{i-1} = \left( \prod_{j=1}^i p_j(D) \right) \psi_0$. 

Now let us consider $p_{a+1}(x)$, that is, the first irreducible quadratic polynomial in the factorization of $p(x)$. It seems natural to associate to this term {\em two} elements of the GBF sequence.
In order to explain the procedure,
assume that $p_{a+1}(x) = x^2 - 2 \theta_{a+1} x + \rho_{a+1}^2$ with $\theta_{a+1}^2 < \rho_{a+1}^2$. %, so that $p(x)$ is an irreducible second order polynomial.
Let further $\psi_a(x) =  \left( \prod_{j=1}^a p_j(D) \right) \psi_0$ be the last element of the GBF sequence constructed before.

Consider the fundamental set 
\begin{align*}
h_1(x) = e^{x \theta_{a+1} } \sin\left( x \sqrt{\rho_{a+1}^2 - \theta_{a+1}^2} \ \right),
\qquad
h_2(x) = e^{x \theta_{a+1} } \cos\left( x \sqrt{\rho_{a+1}^2 - \theta_{a+1}^2} \right),
\end{align*}
associated with the operator $p_{a+1}(D)$.

Recall, from formula \eqref{eq:Km2-2} that
\begin{align*}
p_{a+1}(D) \psi_a(x) = D^2 \psi_a(x) - 2 \theta_{a+1} D \psi_a(x) + \rho_{a+1}^2 \psi_a(x) = 
%\frac{\begin{vmatrix} h_1(x) & h_2(x) & f(x) \\ h_1'(x) & h_2'(x) & f'(x) \\ h_1''(x) & h_2''(x) & f''(x) \end{vmatrix}}{\begin{vmatrix} h_1(x) & h_2(x)  \\ h_1'(x) & h_2'(x) \end{vmatrix}} =  
\frac{W_2}{W_1}D \left(\frac{W_1^2}{W_2} D \left(\frac{\psi_a}{W_1} \right) \right)
\end{align*}
with formally $W_1(x) = h_1(x)$, $W_2(x) = {\begin{vmatrix} h_1(x) & h_2(x)  \\ h_1'(x) & h_2'(x) \end{vmatrix}}$. 

We can compute $W_2$ by taking into account formula \eqref{eq:Wm-3}: since $a_{2,1} = -2 \theta_{a+1}$, we obtain
%First, we search for a solution $h_1$ of the linear equation $p(D)f = 0$, i.e., $h_1(x) = e^{bx} \sin(\sqrt{c^2-b^2} x)$. Further, from the equation $W_2'(x) = -(-2b)W_2(x)$, we obtain 
$W_2(x) = e^{ 2(x-x_0) \theta_{a+1} } W_2(x_0)=c e^{ 2 x \theta_{a+1} }$ for some non-zero constant $c$. 

We define
\begin{align*}
\psi_{a+1} =& c\frac{W_1^2}{W_2} D \left(\frac{\psi_a}{W_1} \right) = \sin^2\left(x \sqrt{\rho_{a+1}^2 - \theta_{a+1}^2} \right) \, D \left(\frac{\psi_a}{e^{x \theta_{a+1} } \sin\left( x \sqrt{\rho_{a+1}^2 - \theta_{a+1}^2} \right)} \right), 
\\ 
\psi_{a+2} =& \frac{W_2}{c W_1}D \psi_{a+1} = \frac{1}{e^{-x\theta_{a+1} } \sin\left( x \sqrt{\rho_{a+1}^2 - \theta_{a+1}^2} \right)} D \psi_{a+1}.
\end{align*}
Note that from
\begin{align*}
\frac{\psi_{a+1}(x)}{\sin^2\left( x \sqrt{\rho_{a+1}^2- \theta_{a+1}^2} \right)} = D \left(\frac{\psi_{a}(x)}{e^{x \theta_{a+1} } \sin\left(x \sqrt{\rho_{a+1}^2- \theta_{a+1}^2} \right)} \right)
\end{align*}
it follows that if we join the set of zeros of the function $\sin\left( x \sqrt{\rho_{a+1}^2- \theta_{a+1}^2} \right)$ with the set of sign-changing zeros of $\psi_{a+1}$ we will get a simple grid for $\psi_a$,
from which the sign-changing zeros of $\psi_a$ can be deduced by a bracketing procedure.%\footnote{?more explanations?}

Moreover, again since
\begin{align*}
\psi_{a+2} = \frac{W_2}{cW_1}D \psi_{a+1},
\end{align*}
it follows that a simple grid for $\psi_{a+1}$ is given by the sign-changing zeros of $\psi_{a+2}$ and the zeros of the function $\sin\left( x \sqrt{\rho_{a+1}^2- \theta_{a+1}^2} \right)$.
Finally, we remark that
\begin{align*}
\psi_{a+2}(x) = p_{a+1}(D)\psi_a(x).
\end{align*}
Note that this implies that $\psi_{a+2}$ has no poles as  $\psi_{a}$ is real analytic.

Continuing on this way, we obtain a sequence of functions $\psi_{a+1}, \dots, \psi_{n}(x)$. Notice that $n = a + 2b$ and 
\begin{align*}
\psi_n(x) = \prod_{j=1}^b p_{a+j}(D) \prod_{i=1}^a p_i(D) f = p(D) f = 0
\end{align*}
and $\displaystyle \{\psi_j,\ j=0, \dots, n\}$ forms a GBF sequence for $f$.

\section{Analysis of Polynomial-Gaussian mixtures}
\label{s:GGM}

In this section, we consider functions of the form
\begin{align}
%\label{eq:bgm}
f(x) = \Re \left[ \sum_{j=1}^n p_j(x) e^{q_j(x)}\right],
\end{align}
where we assume that
\begin{equation}
\text{$q_j$ are real second order polynomials.}
\end{equation}
%This assumption covers the case of unsigned Gaussian mixtures; 
In this case, with no loss of generality, we can also assume that the polynomials $p_j$ are real valued as well and 
we denote $h_j(x) = p_j(x) e^{q_j(x)}$ for $j=1, \dots, n$ and also
%Again, with no loss of generality we assume 
that the functions $\{h_j(x),\ j=1, \dots, n\}$ are linearly independent.
%For simplicity, we shall use the notation
We write
\begin{align}
\label{eq:bgm}
f(x) =  \sum_{j=1}^n p_j(x) e^{q_j(x)} = \sum_{j=1}^n h_j(x).
\end{align}
We will call functions of this type Polynomial-Gaussian mixtures (PGMs). Note that they can also be viewed as mixtures of Hermite functions.

\subsection{GBF sequence for PGMs}

For arbitrary $m \le n$, let us consider the Wronskian determinant $W_m$.
%where we let $h_j(x) = p_j(x) e^{q_j(x)}$ for $j=1, \dots, m$.
Notice that, given a function $h_j$,  %= p_j e^{q_j}$, 
its derivatives of all orders are of the same form:
\begin{align*}
D^k h_j(x) = p_{j,k}(x) e^{q_j(x)}
\end{align*}
where $p_{j,k}(x)$ is a polynomial whose coefficients can be explicitly computed. Since
\begin{align*}
W(h_1, \dots, h_m) = \begin{vmatrix} h_1 & h_2 & \dots & h_m \\ Dh_1 & Dh_2 & \dots & Dh_m \\ \dots & & & \dots
\\ D^{m-1}h_1 & D^{m-1}h_2 & \dots & D^{m-1}h_m \end{vmatrix}
\end{align*}
it follows that the Wronskian $W_m$ has the form
\begin{align*}
W_m(x) = P_m(x) e^{\sum_{j=1}^m q_j(x)}
\end{align*}
where $P_m$ is a polynomial and $P_m \not\equiv 0$. 
As we can determine all the zeros of a polynomial within the interval $[a,b]$ (e.g., using the Budan-Fourier method, or the Sturm chain method, etc.), the Polya-Ristroph formula gives us a GBF sequence:
\begin{equation}\label{eq:gbf-seq1}
\begin{aligned}
\psi_0 &= f = \sum_{j=1}^n p_j(x) e^{q_j(x)}
\qquad
&\psi_1 &= W_1 D\left( \frac{\psi_0}{W_1} \right)
\\
\psi_2 &= \frac{W_2}{W_1} D \left( \frac{W_1}{W_2} \psi_1 \right)
\qquad
&\dots
\\
%\psi_{j+1} &= \frac{W_{j+1}}{W_j} D \left( \frac{W_j}{W_{j+1}} \psi_j \right)\\
& \dots
\qquad
&\psi_n&= \frac{W_{n}}{W_{n-1}} D \left( \frac{W_{n-1}}{W_{n}} \psi_{n-1} \right) 
\end{aligned}
\end{equation}
%\begin{equation}\label{eq:gbf-seq1}
%\begin{aligned}
%\psi_0 &= f = \sum_{j=1}^n p_j(x) e^{q_j(x)}
%\\
%\psi_1 &= W_1 D\left( \frac{\psi_0}{W_1} \right)
%\\
%\psi_2 &= \frac{W_2}{W_1} D \left( \frac{W_1}{W_2} \psi_1 \right)
%\\
%\vdots\\
%\psi_{j+1} &= \frac{W_{j+1}}{W_j} D \left( \frac{W_j}{W_{j+1}} \psi_j \right)\\
%\vdots\\
%\psi_n&= \frac{W_{n}}{W_{n-1}} D \left( \frac{W_{n-1}}{W_{n}} \psi_{n-1} \right) 
%\end{aligned}
%\end{equation}
Notice that $\frac{W_j}{W_{j+1}} = \frac{P_j}{P_{j+1}} e^{q_{j+1}}$ and its inverse $\frac{W_{j+1}}{W_j}$ are rational functions times an exponential, of which we can determine all the zeros and poles.
Further,
$\psi_{n} \equiv 0$
as $f$ is a solution to the linear differential equation
\begin{align*}
\frac{W(h_1,\dots, h_n,f)}{W(h_1,\dots, h_n)} = 0.
\end{align*}
So indeed $\psi_0 = f, \psi_1, \dots, \psi_n \equiv 0$ is a GBF sequence for $f$. 

\subsection{An alternate approach}

Even if the above method is mathematically elegant, its application requires  the computation of high order derivatives and determinants of large matrices. 
Therefore, the method could seem quite elaborate to understand at a first glance. 
In this regard, we propose an alternate, more intuitive approach for the construction of the GBF sequence with the aim of providing the reader with a clearer structural framework of the problem. 
As shown below, the approach is based on an iterative procedure that requires at each step the computation of only first order derivatives rather than Wronskians of arbitrary order.

However, it is to be noted that the new method, even if theoretically more feasible, requires multiple computations of polynomial quotients which makes its implementation even more challenging to be handled. 
In fact, the coefficients of these polynomials stem from iterative computations, therefore common factors detection needs to be addressed properly. 
One way could be to compute symbolical expressions of the polynomials, take out common factors and store the results in a database which could then be adapted to different implementations by just substituting symbolic values of the coefficients with numerical ones. 
However, this approach was found to be computationally unfeasible for larger values of $n$ due to symbolic polynomials division limitations. 
Therefore, deeming the approach more robust and self-contained even if more elaborate, we decided to implement the computation in \cref{sec:ggm} through the GBF sequence with the Wronskians approach. 
That said, the alternate approach is theoretically valid and offers a clearer view of the mechanism undergoing the building process of the GBF sequence as we shall show in the example provided in \cref{sec:ex-pgm}.

Recall the representation in \eqref{eq:bgm}
\begin{align*}
f(x) =  \sum_{j=1}^n p_j(x) e^{q_j(x)} = \sum_{j=1}^n h_j(x).
\end{align*}
We write $\psi_1$ in the following form 
\begin{align*}
\psi_1(x) = h_1(x) D \left( \frac{f(x)}{h_1(x)} \right) = h_1(x) \, \sum_{j=2}^n D \left( \frac{h_j(x)}{h_1(x)} \right) 
= \sum_{j=2}^n \underbrace{\left( h_j'(x) - h_j(x) \frac{h_1'(x)}{h_1(x)} \right)}_{h_{1;j}}
%=\colon \sum_{j=2}^n h_{1;j}
\end{align*}
A simple grid for $\psi_0 = f$ is given by the zeros and the poles of $\psi_1$ and those of $h_1$ (i.e., those of $p_1$).
It shall be noticed that the functions $h_{1;j}(x)$ have a peculiar form
\begin{align*}
    h_{1;j}(x) = r_{1;j}(x) e^{q_j(x)},
\end{align*}
where $r_{1;j}(x)$ is a rational function whose poles correspond to (a subset of) the zeros of $h_1$.

%Since we are interested in the construction of a GBF sequence, we need to find $\psi_2$ in such a way that a simple grid for $\psi_1$ can be obtained from the sign-changing zeros of $\psi_2$ (and that the procedure finishes in a finite number of steps!).

Since we are interested in the construction of a GBF sequence, we 
write
%Therefore, we \emph{restart} the procedure, by choosing
\begin{align*}
\psi_2(x) = h_{1;2}(x) D \left( \frac{\psi_1(x)}{h_{1;2}(x)} \right).
\end{align*}
A direct computation shows that this function coincides with the second term in the GBF sequence defined before.
Therefore, just as before, a simple grid for $\psi_1$ is given by the zeros and the poles of $\psi_2$ and those of $r_{1;2}$ (the rational factor of $h_{1;2}$).
However, this new formulation shows that while $\psi_1$ is defined by the sum of $n-1$ terms, now $\psi_2$ is given by the sum of $n-2$ terms, that is, with a consistent notation:
\begin{align*}
\psi_2(x) =  \sum_{j=3}^n \left( h_{1;j}'(x) - h_{1;j}(x) \frac{h_{1;2}'(x)}{h_{1;2}(x)} \right)
= \sum_{j=3}^n h_{2;j}(x).
\end{align*}
therefore, the procedure finishes in $n$ steps, like before.
We can identify the pivot functions as $\varrho_1 = h_1$ and $\varrho_k = h_{k-1;k}$, $k = 2, \dots, n$.
A worked-out example of this procedure is given in  \cref{sec:ex-pgm}.

%\newpage

\section{Analysis of the class of  Gaussian Mixtures with arbitrary coefficients}
\label{sec:ggm}

In this section we use the Polya-Ristroph formula to construct a GBF sequence 
for finite Gaussian mixtures having coefficients of arbitrary sign
\begin{equation}
\label{eq:ggm1}
%\exists \; N \in \mathbb{N} : 
f(x) = \sum_{k=1}^{n} \gamma_k h_k(x) \qquad   x \in \mathbb{R}
\end{equation}
where $n$ is given, and
\begin{itemize}
	\item[i)] $\displaystyle h_k(x) =  e^{-q_k(x)}$ where $q_k(x)$ is a second degree, non positive polynomial $\displaystyle q_k(x) = -\frac{1}{2 \sigma_k^2} (x-\mu_k)^2 $ with $\mu_k \in \mathbb{R}, \, \sigma_k^2 > 0$ for all $1 \leq k \leq n $;
	\item[ii)] for all $1 \leq k \leq n $, $\gamma_k \, \in \mathbb{R}$ satisfies  $\displaystyle \sum_{k=1}^n  \sigma_k \gamma_k = \frac{1}{\sqrt{2 \pi}}$.
\end{itemize}
By definition, a function of the form \eqref{eq:ggm1} is a probability density function provided that
\begin{itemize}
	\item[iii)] $ f(x) \geq 0 \quad \forall \; x  \in \mathbb{R}$.
\end{itemize}
%%%%%%%%%%%%%%%%%%%%%%%%%%%%%%%%

We shall refer to functions of the form \eqref{eq:ggm1} as \emph{finite Gaussian Mixtures }  (see the discussion in  \cref{sec:intro} for a introduction to their properties and main applications in the literature).

\begin{remark}
\label{rem:hermite}
Recall that for the standard Gaussian density we define the Hermite polynomials as
\begin{align*}
H_0(x) = 1, \qquad H_n(x) = (-1)^n e^{x^2/2} D^n e^{-x^2/2}, \qquad n = 1, 2, \dots.
\end{align*}
%let $H_n(x, \lambda) = \lambda^{n/2} H_n(x/\sqrt{\lambda})$.
Then, for $q(x) = \frac1{2\sigma^2}(x-\mu)^2$ it holds
\begin{align*}
D^n e^{-q(x)} = %(-1)^n e^{-q(x)} H_n(Dq(x), D^2 q(x)) = 
\frac{(-1)^n}{\sigma^n} e^{-q(x)} H_n\left( \frac{x-\mu}{\sigma}\right).
\end{align*}
Let us denote
\begin{align*}
P_n(x) = \frac{(-1)^n}{\sigma^n}  H_n\left( \frac{x-\mu}{\sigma}\right).
\end{align*}
Then we know that $P_n(x)$ is a polynomial of degree $n$ that satisfies the recurrence relation
\begin{align*}
P_{n+1}(x) = - \frac{x-\mu}{\sigma^2} P_n(x) +  D P_n(x). 
\end{align*}
\end{remark}

\vskip 1\baselineskip

With no loss of generality, we assume that $h_i \neq h_j \; \, \forall \, i \neq j $ since any mixture can be traced back to this case. 
Notice that for any given $h_k(x)$ its derivatives of all orders are of the form 
\begin{align*}
    D^m h_k(x) = P_{m,k}(x) h_k(x)
\end{align*}
where $P_{m,k} $ is the generalized $m$-th degree Hermite polynomial 
with coefficients $\mu_k$ and $\sigma^2_k$:
\begin{align*}
P_{m,k}(x) = \frac{(-1)^m}{\sigma_k^m}  H_m\left( \frac{x-\mu_k}{\sigma_k}\right), \qquad m = 1, 2, \dots.
\end{align*}
Since they play an important r\^ole in the sequel, we introduce the linear functions
\begin{equation}
\label{eq:phi}
\phi_k(x) = D q_k(x) = P_{1,k}(x) = \frac{\mu_k - x }{\sigma^2_k}, \qquad k=1, \dots, n,
\end{equation}
so for instance the recurrence relation reads
\begin{align*}
P_{m+1,k} = \left( \phi_k + D \right) P_{m,k}(x), \qquad m = 0, 1, \dots, k = 1, \dots, n.
\end{align*}
We can further express the $m$-th order Wronskian for the sequence $\{h_j\}_{j=1}^{n}$ as
\begin{align*}
W(h_1, \dots, h_m) = \left( \prod_{j=1}^m e^{- q_j } \right) \, \begin{vmatrix} 1 & 1 & \dots & 1 \\ P_{1,1} & P_{1,2} & \dots & P_{1,m} \\ \dots & & & \dots
\\ P_{m-1,1} & P_{m-1,2} & \dots & P_{m-1,m} \end{vmatrix}.
\end{align*}

For any $j \ge 1$ and for any sequence of increasing integers $i_1, \dots, i_j$, we introduce the polynomial
\begin{align}
\label{eq:Q}
Q^{j-1}_{i_1, \dots, i_j} = \begin{vmatrix} 1 & 1 & \dots & 1 \\ P_{1,i_1} & P_{1,i_2} & \dots & P_{1,i_j} \\ \dots & & & \dots
\\ P_{j-1, i_1} & P_{j-1, i_2} & \dots & P_{j-1, i_j} \end{vmatrix}
\end{align}
(and, in particular, $Q^0_k = 1$ for any $k \ge 1$)
so that
\begin{align}
\label{eq:determ}
W(h_1, \dots, h_m) = \left( \prod_{j=1}^m e^{- q_j } \right) Q^{m-1}_{1, \dots, m}.
\end{align}

\begin{theorem}
\label{th:tp1}
The following relations hold:
\begin{equation}
\label{eq:Q1}
Q^1_{j,k} = \phi_k - \phi_j, \qquad 1 \le j \le k \le n
\end{equation}
and
\begin{equation}
\label{eq:Q2}
Q^{j}_{1, \dots, j, k} 
Q^{j-2}_{1, \dots, j-1}
= 
Q^{j-1}_{1, \dots, j} \, \left[  \phi_k  - \phi_j \right] Q^{j-1}_{1, \dots, j-1, k}
+
Q^{j-1}_{1, \dots, j} \, D_x  Q^{j-1}_{1, \dots, j-1, k}
- 
Q^{j-1}_{1, \dots, j-1, k}  \, D_x Q^{j-1}_{1, \dots, j} 
\end{equation}
for $2 \le j < k \le n$.
\end{theorem}

The proof of this result is given in \cref{sec:app-proof-tp1} for the sake of clarity of exposition.
Next, we can use the above relations 
to determine a different formulation of the sequence of Wronskians of the problem, which involves polynomials recursively defined rather than computing derivatives of eventually high orders, which we present in the following result. 
%The last two results allow us to determine an additional formulation of the GBF by which the wronskians abovementioned can be interpreted as recursively-defined polynomials. 

\begin{theorem} \label{firstresult}
Let $\displaystyle f(x) = \sum_{k=1}^n \gamma_k h_k(x)$ be a finite Gaussian mixture. 
Then, the GBF sequence obtained using the Polya-Ristroph formula can be rewritten as 
\begin{equation}
\label{eq:recursive-form}
\psi_j(x) =
	\begin{cases}
	f(x) \qquad & j=0 \\
	\displaystyle
	\sum_{k=j+1}^n \lambda_k \frac{Q^{j}_{1, \dots, j, k} }{ Q^{j-1}_{1, \dots, j} } \, h_k(x) \qquad  &  \forall \; 1 \le j < n. 
	\end{cases}
	\end{equation}
\end{theorem}

The proof of this result is given in the  \cref{sec:app-proof-first-result}. 

\begin{remark}
The conundrum of opting for one choice or another when building the GBF sequence naturally arises for computational purposes. 
In fact, whilst the choice of the Wronskians could lead to high-order derivatives and large matrices, it would also allow a more neat and direct definition of the objects involved in the building process of the GBF sequence. 
On the other hand, even if the polynomial recursive approach appears to be quite faster because only first-order derivatives are involved, it is to be noted that it entails hidden recursive dependencies as long as it requires the execution of polynomial division. 

In the next subsection we will describe the main steps of an algorithm which computes the GBF sequence in accordance with the Wronskians approach. 
This is due, on the one hand, to avoid multiple executions of polynomial divisions which could cause long execution times if performed with symbolic coefficients and incorrect common factors detection in the case of numerical coefficients; on the other hand, to the availability in the literature of a fast algorithm for determinant computation \cite{Kasyanov2022}.
%This is due essentially to avoid multiple executions of polynomial divisions which could cause long execution times if performed with symbolic coefficients and incorrect common factors detection in the case of numerical coefficients.

For the sake of completeness, we also computed the Wronskians using symbolic expressions and stored them in a database (which is available as an attachment to this work) for dimensions spanning from $2$ to $16$. 
This is done in order to provide the reader with general ready-to-use results, and we hope that it may prompt future research in the direction of finding simpler recursive structures for GBF sequences. 
In this regard, we highlight that all our results were proved using formal mathematical procedures, but it appears reasonable to envisage alternative methods as for instance artificial intelligence self-learning algorithms to investigate structure properties of GBF sequences.
\end{remark}

\subsection{Implementation}

We have developed a software in Matlab with the aim of  investigating how the GBF approach works in practice.

The algorithm takes as inputs the following parameters: the number of Gaussians involved in the mixture, their means, variances, the mixing coefficients, and the desired accuracy level for the searching-roots algorithm. It is to be noted that the procedure does not require to be given a bounded interval to test the mixture on because such interval is computed internally. Hence, a searching-roots problem on the real line is reduced to a constrained searching-roots procedure on a bounded interval. As outputs, the software returns the location of the eventual sign-changing roots and the bounded interval in which they are contained.

If the software returns no sign-changing zeros for the Gaussian mixture $f(x)$, and $f(x)$ is positive in at least one point, then it is non-negative everywhere and it defines a proper probability density function.

\subsection*{Technical aspects and execution time}

To further reduce computational time, the classical bisection method is replaced by Ridders' method which proves to be more efficient in searching-roots problems \cite{Ridders}. 
Execution times (in secs) are displayed in  \cref{tab:1}. They were computed by running 100 simulations for each dimension spanning from 2 to 12 where the parameters were chosen as follows:
\begin{itemize}
    \item the means where randomly picked in [-10, 10];
    \item the variances were randomly picked in [0.1, 1];
    \item the mixing coefficients were randomly picked in [-1, 1];
    \item The accuracy level was set equal to $2.2204^{-16}$.
\end{itemize}

\medskip

\begin{table}[!htbp]
	\begin{center}
		\begin{tabular}{|lcccccp{2cm}|}
			\hline
			\textbf{n} & \textbf{nsim} & \textbf{min}  & \textbf{median}  & \textbf{average}  & \textbf{max} & \textbf{\% time GBF sequence}\\
			\hline
			2 & 100 & 0.23 & 0.60 & 0.77 & 1.55 & 28.92 \\
			\hline
			3 & 100 & 0.43 & 1.16 & 1.20 & 3.98 & 29.17 \\
			\hline
			4 & 100 & 0.83 & 2.79 & 2.69 & 5.74 & 28.26 \\
			\hline
			5 & 100 & 1.49 & 5.83 & 5.92 & 11.14 & 25.82 \\
			\hline
			6 & 100 & 5.06 & 10.85 & 11.35 & 29.91 & 24.23 \\
			\hline
			7 & 100 & 7.96 & 22.73 & 22.12 & 38.28 & 23.95 \\
			\hline
			8 & 100 & 21.20 & 42.15 & 42.88 & 62.92 & 37.16 \\
			\hline
			9 & 100 & 50.44 & 91.45 & 91.25 & 183.57 & 36.52 \\
			\hline
			10 & 100 & 111.18 & 177.90 & 177.68 & 256.83 & 49.26 \\
			\hline
			11 & 100 & 338.22 &  473.11 & 486.49 & 922.00 & 68.32 \\
			\hline
			12 & 100 & 1260.00 &  1466.10 & 1533.40 & 2371.40 & 79.03 \\
			\hline
		\end{tabular}
	\vspace{0.2cm}
	\caption{Analysis of execution times for dimensions spanning from $n=2$ to $n = 12$.}
	\label{tab:1}
	\end{center}
\end{table}

Furthermore, average execution times for higher dimensions were estimated by fitting available data by a function of the kind $a e^{bx}+c e^{dx}$ $(a = 0.4517, b = 0.5572, c = 3.61e-05, d = 1.441)$ with the Matlab function fit. The results are shown in  \cref{fig:exec_times}.

\begin{figure}[!htbp]
	\centering\begin{minipage}[b]{0.7\textwidth}
		\includegraphics[width=\textwidth]{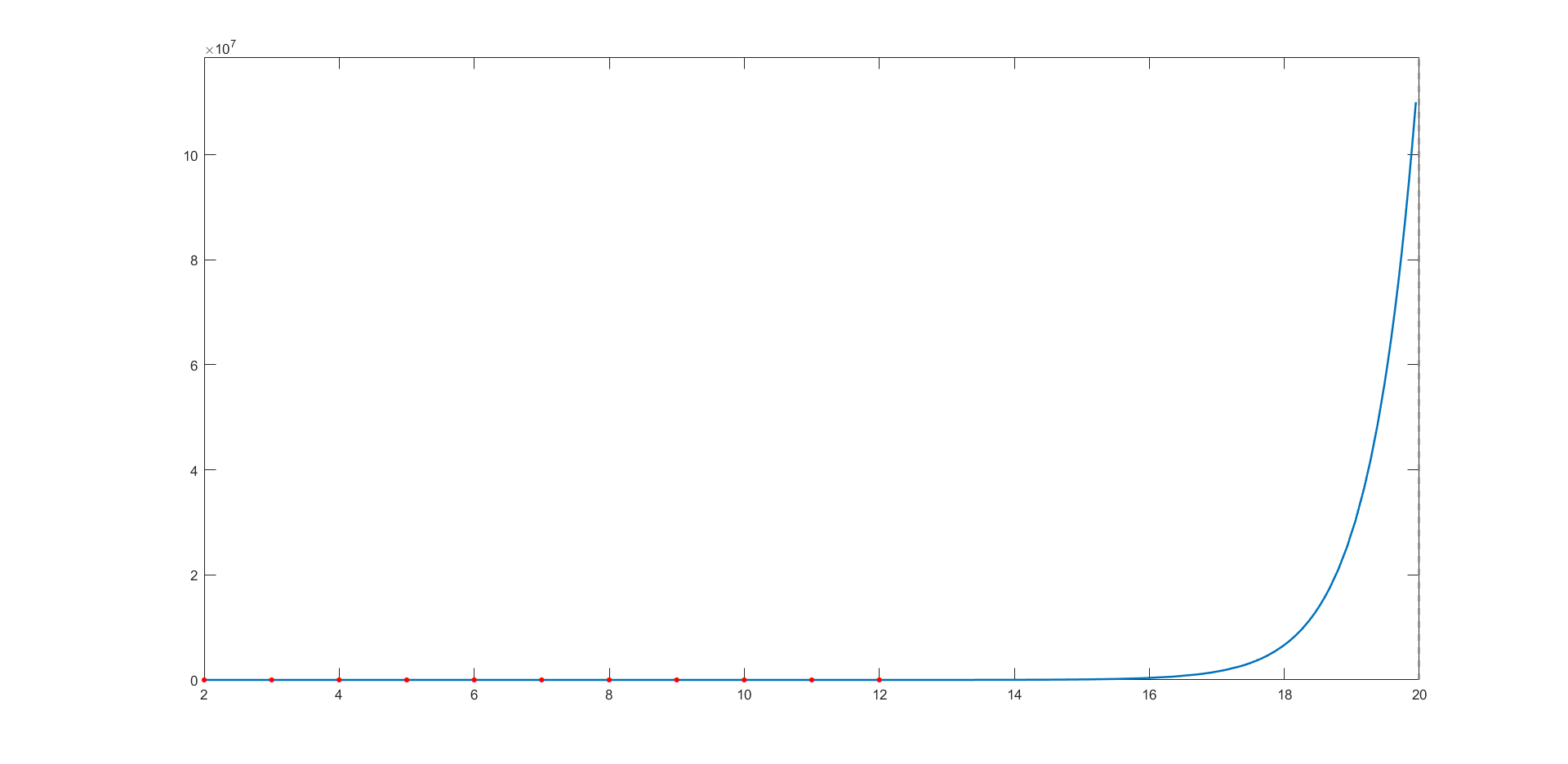}
	\end{minipage}
\caption{Fitted curve for average execution times up to $n=20$.}
\label{fig:exec_times}
\end{figure}

As we can observe from the plot, average execution times tend to significantly increase from $n = 16$, reaching values in the range of $10^7$ secs for $n = 20$. 
The increased execution time is due to the heavier computation of the GBF sequences using the Wronskian approach - as long as to the available hardware - as we can observe from the last column of \cref{tab:1} showing the percentage of time employed in the construction of the GBF sequence. 
Furthermore, percentage data highlight the interplay between the two main building blocks of the algorithm, namely the construction of the GBF sequence and the searching-root procedure, along different dimensions. 
In fact, we can observe that for $ 2 \leq n \leq 4$ the percentage time dedicated to the construction of the GBF sequence stabilizes around $28-29\%$, while it decreases around $24-26\%$ for $ 5 \leq n \leq 7 $ before reaching higher values spanning from $36\%$ to $79\%$ for $ 8 \leq n \leq 12$. 
Therefore, at least at present, we suggest that applications that require high execution times rely on computing systems with large computational power.
%The increased complexity is due to the heavier computation of the GBF sequences using the Wronskian approach. Therefore, at least at present, applications that require high number of Gaussians  rely on computing systems with large computational power.

\section{Wasserstein distance}
\label{s:Was}

In this section, we aim to present a different application where our algorithm can be useful since, as we shall see below, the computation of the Wasserstein distance between two Gaussian mixtures can be reduced to a simple sum, once we know the position of the zeros of the difference of the corresponding cumulative distribution functions. 
We are not interested in discussing the topic of Wasserstein distance in general and we refer to the existing literature for further discussion, see, e.g., \cite{Ambrosio2008}. 
In order to fix the notation, we start with a few definitions.
We let $S \subset \mathbb{R}^n$ denote a general domain, endowed with the Euclidean distance. Below, we usually consider $S= \mathbb{R}$.

\begin{definition}\label{d:1}
Let $M_p(S)$ be the space of all probability measures $\mu$ on $S$ with finite $p$-th moment for some $x_0 \in S$:
\begin{align*}
\int_S |x-x_0|^p \, \mu({\rm d}x) < +\infty.
\end{align*}
Then the Wasserstein-$p$ distance between two probability measures $\mu$ and $\nu$ in $M_p(S)$ is defined as
\begin{align*}
W_p(\mu,\nu) = \left( \inf_\gamma \int_{S \times S} |x-y|^p \, \gamma({\rm d}x,{\rm d}y) \right)^{1/p},
\end{align*}
where the infimum is taken among all two-dimensional measures $\gamma$ with marginals $\mu$ and $\nu$ respectively.
\end{definition}

It is proved in \cite[Theorem 20.1]{Dudley1976} that in the particular case when $p = 1$ we obtain
\begin{align*}
W_1(\mu,\nu) = \sup \left\{ \int \varphi \, {\rm d}(\mu - \nu) \,\mid\, \varphi: S \to \mathbb{R} \quad 1-{\rm Lipschitz}\right\}.
\end{align*}
An alternate form of the $W_1$ distance was given by Dall'Aglio \cite{DallAglio1956}:\footnote{notice that we do not make any assumption on the form of the distribution}
\begin{align}\label{eq:2}
W_1(\mu,\nu) = \int_0^1 \left| F^{-1}(u) - G^{-1}(u)\right| \, {\rm d}u,
\end{align}
where $F$ and $G$ are the cumulative distribution functions for $\mu$ and $\nu$ respectively, and $F^{-1}$, $G^{-1}$ denote the quantile functions of $\mu$ and $\nu$ respectively. 
For the sake of completeness, we recall the definition below.

\begin{remark}
Let us recall the definition of \emph{generalized} inverse (compare \cite{Embrechts2013}) for a distribution function $F$ (also called the 
\emph{quantile function} of $F$).
$F^{-1}: [0,1] \to \bar{\mathbbm{R}} = [-\infty,+\infty]$ is defined by
\begin{align*}
F^{-1}(u) = \inf \left\{x \in \mathbb{R}\,:\, F(x) \ge u\right\}, \qquad u\in [0,1],
\end{align*}
with the convention that $\inf \emptyset = +\infty$. By this definition, the 0-quantile of $F$ is always $F^{-1}(0) = -\infty$.
\end{remark}

For our purposes, the following result is the best way to treat the distance between one-dimensional distributions.
Assume that the laws $\mu$ and $\nu$ have cumulative distribution functions $F$ and $G$, respectively. Then the following result is classical \cite{Dudley1976, Vallander1974}.

\begin{theorem}
The Wasserstein-1 distance between $\mu$ and $\nu$ in $M_1(\mathbb{R})$ (see  \cref{d:1}) is equal to the $L^1$-distance between the corresponding cumulative distribution functions $F$ and $G$:
\begin{align}\label{eq:1}
W_1(\mu,\nu) = \int_{\mathbb{R}} |F(x) - G(x)| \, {\rm d}x.
\end{align}
\end{theorem}

The above formula is well defined in $M_1(\mathbb{R})$; for completeness (and since it seems a little tricky) we prove it below. 

\begin{lemma}
Assume that $\mu$ and $\nu$ belong to $M_1(\mathbb{R})$. Then $w(x) = F(x) - G(x)$ is absolutely integrable.
\end{lemma}

\begin{proof}
We write
\begin{multline*}
W_1(\mu, \nu) = \int_{-\infty}^0 |F(x) - G(x)| \, {\rm d}x + \int_0^{+\infty} |(1-F(x)) - (1-G(x))| \, {\rm d}x
\\
\le \int_{-\infty}^0 \left\{ |F(x)| + |G(x)| \right\} \, {\rm d}x + \int_0^{+\infty} \left\{ |1-F(x)| + |1-G(x)| \right\} \, {\rm d}x.
\end{multline*}
It is sufficient to prove that the following quantity is finite for arbitrary $\mu$:
\begin{multline*}
\int_{-\infty}^0 F(x) \, {\rm d}x + \int_0^{+\infty} (1-F(x)) \, {\rm d}x
= \int_{-\infty}^0 \int_{-\infty}^x \mu({\rm d}y) \, {\rm d}x + \int_0^{+\infty} \int_x^\infty \mu({\rm d}y) \, {\rm d}x
\\
= \int_{-\infty}^0 \int_{y}^0  \, {\rm d}x \, \mu({\rm d}y) + \int_0^{+\infty} \int_0^y  \, {\rm d}x \, \mu({\rm d}y)
= \int_{-\infty}^0 |y|  \, \mu({\rm d}y) + \int_0^{+\infty} |y|  \, \mu({\rm d}y) = \mathbb{E}_\mu[|y|]
\end{multline*}
and the last quantity is finite by assumption.
\end{proof}

In the following result we propose an analog of Markov-Chebyshev's inequality for cumulative distribution functions. 
Here, we need to impose the existence of a finite $p$-moment, $p > 1$, for the probability density function.

\begin{lemma}
Let $X$ be an absolutely continuous random variable with probability density function $f(x)$, such that  the finite $p$-moment exists for some $p > 1$. Then for any $L > 0$:
\begin{align}\label{eq:gen-cheb}
\int_L^\infty (1-F(x)) \, {\rm d}x \le \frac{1}{p-1} \, \frac{\mathbb{E}[|X|^p]}{L^{p-1}}.
\end{align}
\end{lemma}

\begin{proof}
We may compute
\begin{multline*}
\int_L^\infty \int_x^\infty f(y) \, {\rm d}y \, {\rm d}x \le \int_L^\infty \int_x^{\infty} \frac{|y|^p}{x^p} f(y) \, {\rm d}y \, {\rm d}x 
\\
\le  \int_L^\infty x^{-p} \left( \int_{-\infty}^\infty |y|^p f(y) \, {\rm d}y  \right) \, {\rm d}x 
= \frac{1}{p-1} \frac{1}{L^{p-1}} \mathbb{E}[|X|^p]
\end{multline*}
as required.
\end{proof}

A similar computation gives a bound for the integral on the interval $(-\infty,-L)$ for the distribution function $F(x)$.
We can therefore apply this inequality to the difference $w(x)$.

\begin{lemma}
Let $X$, $Y$ be random variables with probability density functions $f(x)$, $g(x)$ and  distribution functions $F(x)$ and $G(x)$, respectively, such that their finite $p$-moments exist for some $p > 1$. 
Then there exists a constant $C$, depending on $p$ and the $p$-th moments of $X$ and $Y$,
such that
\begin{align}\label{eq:gen-cheb-diff}
\int_{\mathbb{R} \setminus [-L,L]} |w(x)| \, {\rm d}x \le C \frac{1}{L^{p-1}}
\end{align}
where $w(x) = F(x) - G(x)$
\end{lemma}

\begin{proof}
On the one hand, it is sufficient to compute
\begin{align*}
\int_L^\infty |w(x)| \, {\rm d}x = \int_L^{+\infty} |(1-F(x)) - (1-G(x))| \, {\rm d}x \le \frac{1}{p-1} \frac{\mathbb{E}[|X|^p + |Y|^p]}{L^{p-1}}
\end{align*}
Since a similar bound holds in the negative semi-axes, the thesis follows.
\end{proof}

\subsection{An algorithm for the computation of Wasserstein-1 distance for EPT functions}

How to compute the quantity in \eqref{eq:1}? Our approach will be to consider the difference $w(x) = F(x) - G(x)$ and the sign-changing zeros of $w(x)$.
A nice feature of EPT functions is that they form an algebra, and further that the cumulative density function $F(x)$ of an EPT probability distribution function $f(x) = c e^{A x}b$ is again an EPT function, with cumulative distribution function $F(x) = 1 + c A^{-1} e^{A x} b$. It follows that for EPT functions $F$ and $G$, the difference $w(x)$ is an EPT function as well.

\begin{remark}\label{ass2}
Notice that an EPT function, defined in \eqref{e:ept}, is an entire function, hence all zeros of $w$ are isolated and $w$ is a continuous function.
% (since we deal with EPT distributions, 
%we are ruling out the possibility of positive mass in the 0).
In particular, on any interval $[a,b]$ there will be at most finitely many sign-changing zeros.
\end{remark}

Let us assume  that there are at most a finite number of sign-changing zeros of $w$, say $\xi_1, \dots, \xi_n$, on the whole positive half-line, and that these have been identified using the GBF method. We use the convention $\xi_0 = 0$ and $\xi_{n+1} = +\infty$.
Then we can compute
\begin{align*}
\int_{0}^{\infty} |F(x) - G(x)| \, {\rm d}x =
\sum_{j=0}^n \int_{\xi_j}^{\xi_{j+1}} \left(F(x) - G(x) \right) \mathop{\rm sgn}\left(F(x) - G(x)\right) \, {\rm d}x.
\end{align*}
On each of the intervals $(\xi_j,\xi_{j+1})$ the sign is constant so it can be taken out of the integral. Further,
integrals of the form
\begin{align*}
\int_{\xi_j}^{\xi_{j+1}} \left(F(x) - G(x) \right) \, {\rm d}x
\end{align*}
can be calculated explicitly in  the class of EPF functions.
In fact we have
\begin{align*}
\int_{\xi_j}^{\xi_{j+1}} \left( F(x) - G(x) \right) \, {\rm d}x
= c_1 A_1^{-2}\left( e^{A_1 \xi_{j+1}} - e^{A_1 \xi_{j}} \right) b_1 - c_2 A_2^{-2}\left( e^{A_2 \xi_{j+1}} - e^{A_2 \xi_{j}} \right) b_2.
\end{align*}
Since $A_1$ as well as $A_2$ are continuous-time asymptotically stable, meaning that their spectra satisfy $\sigma(A_i) \subset \{z \in \mathbb{C} \,:\, \Re(z) < 0\}$,
it follows that 
\begin{align*}
\int_{\xi_n}^{\xi_{n+1} = +\infty} \left( F(x) - G(x) \right) \, {\rm d}x = - c_1 A_1^{-2} e^{A_1 \xi_n}  b_1 + c_2 A_2^{-2} e^{A_2 \xi_n}  b_2.
\end{align*} 
So we obtain the following formula for the Wasserstein-1 distance between $\mu$ and $\nu$:
\begin{multline*}
W(\mu,\nu) = \sum_{j=0}^{n-1} \left| c_1 A_1^{-2}\left( e^{A_1 \xi_{j+1}} - e^{A_1 \xi_{j}} \right) b_1 - c_2 A_2^{-2}\left( e^{A_2 \xi_{j+1}} - e^{A_2 \xi_{j}} \right) b_2 \right| \\ + \left| c_1 A_1^{-2} e^{A_1 \xi_n}  b_1 - c_2 A_2^{-2} e^{A_2 \xi_n}  b_2 \right|,
\end{multline*}
where again we  set $\xi_0 = 0$.

\subsection{An algorithm for the computation of  Wasserstein-1 distance for Gaussian mixtures}

Here we ask the same question: \emph{how to compute the quantity in \eqref{eq:1}} in the case of finite Gaussian mixtures?
Again, the two crucial issues are:
\begin{itemize}
\item to compute the sign-changing zeros of the difference of the cumulative distribution functions involved;
\item to calculate the integrals of the difference of the cumulative distribution functions involved, over intervals on which the sign of the difference does not change.
\end{itemize}

In this case, both the distribution functions involved and their differences are of the form
\begin{align*}
w(x) = \sum_{i=1}^K \mu_i \Phi(\alpha_i x + \beta_i),
\end{align*}
where $\Phi$ is the standard Gaussian distribution function. 

A GBF sequence is obtained for $w$ by taking
$w_0(x) = w(x)$, $w_1(x) = Dw(x)$. Then $w_1(x)$ is a function of the form
\begin{align*}
w_1(x) = \sum_{i=1}^K \mu_i \alpha_i \phi(\alpha_i x + \beta_i),
\end{align*}
where $\phi(x) = \frac{1}{\sqrt{2 \pi}} e^{-x^2/2}$.
\\
We have already developed a GBF sequence for $w_1$ in  \cref{sec:ggm}, say $\psi_0 = w_1, \psi_1, \dots, \psi_K$. 
Then a GBF sequence for $w$ is given by $w_0 = w, w_1, \psi_1, \dots, \psi_K$. 

\begin{lemma}
A Gaussian mixture  $\displaystyle f = \sum_{i=1}^K \delta_i \phi(\alpha_i x + \beta_i)$, and hence also $w$, has at most finitely many sign-changing zeros and an interval can be identified containing all these.
\end{lemma}

\begin{proof}
We emphasise that this result holds for generalized Gaussian mixtures with a similar proof. 
We aim to prove that there exists $x_0 > 0$ such that, for all $x > x_0$, $f(x) \not= 0$. 
Notice that
\begin{align*}
f(x) = \sum_{i=1}^K \delta_i \phi(\alpha_i x + \beta_i) = \frac{1}{\sqrt{2\pi}} \sum_{i=1}^K \delta_i e^{-\beta_i^2/2} e^{-q_i(x)}, \qquad q_i(x) = {\frac{\alpha_i^2}{2} x^2 + \alpha_i \beta_i x}
\end{align*}
can be ordered in such a way that 
\par\smallskip\noindent\begin{minipage}{\textwidth}
\begin{itemize}
\item $\alpha_i < \alpha_{i+1}$, or 
\item $\alpha_i = \alpha_{i+1}$ and $\beta_i < \beta_{i+1}$
\end{itemize}    
\end{minipage}
\par\smallskip\noindent (a further equality cannot hold, since otherwise the two polynomials would have been equal). 
It follows that
\begin{align*}
f(x) = \frac{1}{\sqrt{2 \pi}} e^{-q_1(x)} \left(\delta_1 e^{-\beta_1^2/2} + \sum_{i=2}^K \delta_ie^{-\beta_i^2/2}  e^{q_1(x)-q_i(x)} \right)
\end{align*}
has the property that the polynomial $q_1 - q_i$ has the leading (non-zero) coefficient {\it negative} for each $i = 2, \dots, K$: therefore, for each $i = 2, \dots, K$,
there exists $x_i$ such that $\left| \delta_i  e^{q_1(x)-q_i(x)}  \right| \le \frac{1}{2K} \left|  \delta_1\right|$ for all $x > x_i$.
Setting $x_0 = \max(x_2, \dots, x_K)$ we obtain the claim.

Now we claim that it is possible to find $y_0 < 0$ such that, for all $x < y_0$, $f(x) \not= 0$. The proof is the same as before, except that we shall modify the ordering of the polynomials to take into account the correct behaviour of the first order term. Actually we choose
\par\smallskip\noindent\begin{minipage}{\textwidth}
\begin{itemize}
\item $\alpha_i < \alpha_{i+1}$, or 
\item $\alpha_i = \alpha_{i+1}$ and $\beta_i > \beta_{i+1}$:
\end{itemize}
\end{minipage}\par\smallskip\noindent
then the proof follows with the same argument as above.
\end{proof}

To conclude this section, we recall that the integral of the Gaussian cumulative distribution function can be computed explicitly in terms of the Gaussian distribution itself.
The starting point is the following well known identity:
\begin{align}\label{eq*}
\int_{-\infty}^x \Phi(y) \, {\rm d}y = x \Phi(x) + \phi(x),
\end{align}
where $\Phi$ is the cumulative distribution function and $\phi$ the probability density function for the standard Gaussian distribution.
Then, from \eqref{eq*} it follows that
\begin{multline*}
\int_a^b \Phi(\alpha x + \beta) \, {\rm d}x =
% y = \alpha x + \beta
\int_{\alpha a + \beta}^{\alpha b + \beta} \frac{1}{\alpha} \Phi(y) \, {\rm d}y 
\\
= \left( b + \frac{\beta}{\alpha} \right) \Phi \left(\alpha b + {\beta} \right) 
-  \left( a + \frac{\beta}{\alpha} \right)  \Phi \left( \alpha a + \beta \right) 
+ \frac{1}{\alpha} \left[ \phi\left( \alpha b + \beta \right)  - \phi\left( \alpha a + \beta \right)  \right].
\end{multline*}

\section{Conclusion and further research}

In this paper we explored finite mixture models and we concentrated our effort on two particular subclasses of \eqref{e1}, Gaussian mixtures and EPT functions, that have more relevance in probability. We hope to return to the general treatment of general mixtures defined in \eqref{e1} in a following paper.
\\
Also, our results are specifically stated in one space dimension. In several applications, this is quite restrictive, and the extension of our results to the $d$-dimensional case is an important topic for further research.

%\vfill\newpage

\bigskip

\appendix
\section{Technical results and proofs}

\subsection{Some results in bordered matrices}

Let 
\begin{align*}
M = \begin{pmatrix} \tilde M & c_1 & c_2 \\ b_1^T & a_1 & a_2 \\ b_2^T & a_3 & a_4 \end{pmatrix}
\end{align*}
be a square $n \times n$ matrix, partitioned in such a way that
\begin{itemize}
\item $\tilde M$ is a $(n-2) \times (n-2)$ square matrix,
\item $c_1$, $c_2$, $b_1$, $b_2$ are column vector of length $n-2$, and
\item $a_1, \dots, a_4$ are numbers.
\end{itemize}
Then we can compute the determinant of $M$
and we have the following rule (see e.g.\ \cite{Karapiperi2015})
\begin{equation}
\label{eq:detM.detQ}
{\rm det}(M) \, {\rm det}(\tilde M) = \begin{vmatrix} \begin{vmatrix} \tilde M & c_1 \\ b_1^T & a_1 \end{vmatrix} &
\begin{vmatrix} \tilde M & c_2 \\ b_1^T & a_2 \end{vmatrix} \\ \\
\begin{vmatrix} \tilde M & c_1 \\ b_2^T & a_3 \end{vmatrix} & \begin{vmatrix} \tilde M & c_2 \\ b_2^T & a_4 \end{vmatrix} 
\end{vmatrix}
\end{equation}
and, if we set
\begin{align*}
A' = \begin{pmatrix} \tilde M & c_1 \\ b_1^T & a_1 \end{pmatrix},
\qquad 
B' = \begin{pmatrix} \tilde M & c_2 \\ b_1^T & a_2 \end{pmatrix}, \qquad
C' = \begin{pmatrix} \tilde M & c_1 \\ b_2^T & a_3 \end{pmatrix}, \qquad
D' = \begin{pmatrix} \tilde M & c_2 \\ b_2^T & a_4 \end{pmatrix} 
\end{align*}
we obtain
\begin{align}
\label{eq:detM.detQ-bis}
{\rm det}(M) \, {\rm det}(\tilde M) = {\rm det}(A') \, {\rm det}(D') - {\rm det}(B') \, {\rm det}(C').
\end{align}

\subsection{Proof of  \cref{th:tp1}}\label{sec:app-proof-tp1}

Relation \eqref{eq:Q1} follows directly from the definition
\begin{align*}
Q^1_{i,k} =  \begin{vmatrix} 1 & 1  \\ P_{1,i} & P_{1,k} \end{vmatrix}
\end{align*}
and the definition \eqref{eq:phi}.

Next we consider the case of $j > 2$.
With a slight abuse of notation, we denote by
\begin{align*}
Q^{j-1}_{i_1, \dots, i_j}(y_1, \dots, y_j)
\end{align*}
the determinant in \eqref{eq:Q} where in each column we consider a different variable $y_i$. 
Notice that we recover the polynomial in $x$ when we let $y_1 = \dots = y_j = x$:
\begin{align*}
Q^{j-1}_{i_1, \dots, i_j}(x) = Q^{j-1}_{i_1, \dots, i_j}(x, \dots, x).
\end{align*}

Let $D_i$ denote the differentiation with respect to the variable $y_i$.
Then the following formula holds (chain rule)
\begin{align}
\label{eq:derivate.uguali}
D_x Q^{j-1}_{i_1, \dots, i_j}(x) = [D_{1} + \dots + D_j]Q^{j-1}_{i_1, \dots, i_j}(x, \dots, x).
\end{align}

Finally, we introduce the notation
\begin{align*}
\Delta_j = \phi_j + D_j,
\end{align*}
where $\phi_j$, defined in \eqref{eq:phi},  is taken as a multiplication operator.
By the results in  \cref{rem:hermite} we know that
\begin{align*}
\Delta_j P_{m,j} = P_{m+1,j}
\end{align*}
and
\begin{equation}\label{eq:derivata}
\left[\sum_{i=1}^j \Delta_i \right] Q^{j-1}_{i_1, \dots, i_j} = 
\begin{vmatrix} 1 & 1 & \dots & 1 \\ P_{1,i_1} & P_{1,i_2} & \dots & P_{1,i_j} \\ & \dots & & \dots
\\ P_{j-2, i_1} & P_{j-2, i_2} & \dots & P_{j-2, i_j} 
\\ P_{j, i_1} & P_{j, i_2} & \dots & P_{j, i_j}.
\end{vmatrix}
\end{equation}

Now, let us write $Q^{j}_{1, \dots, j, k}$ as the determinant of a bordered matrix in the following form
\begin{align*}
Q^{j}_{1, \dots, j, k}
=  
\left|\begin{array}{ccc|c|c} 
1 & \dots & 1 & 1 & 1 
\\ 
P_{1,1} & \dots & P_{1, j-1}  & P_{1,j} & P_{1,k}
\\
& \dots & & & \dots
\\ 
P_{j-2, 1} & \dots & P_{j-2, j-1} & P_{j-2,j} & P_{j-2, k} 
\\
\hline
P_{j-1, 1} & \dots & P_{j-1,j-1} & P_{j-1,j} & P_{j-1, k} 
\\
\hline
P_{j, 1} & \dots & P_{j,j-1} & P_{j,j} & P_{j, k} 
\end{array}\right|.
\end{align*}
Our aim is to apply formula \eqref{eq:detM.detQ-bis}.
It is clear that
\begin{align*}
{\rm det}(M) = Q^{j}_{1, \dots, j, k}
\end{align*}
as well as
\begin{align*}
{\rm det}(\tilde M) = Q^{j-2}_{1, \dots, j-1};
\end{align*}
it remains to identify the other terms.
Recall the definition of matrices $A', B', C'$ and $D'$ from previous subsection.
We have 
\begin{align*}
A' &=
\left|\begin{array}{ccc|c} 
1 & \dots & 1 & 1 
\\ 
P_{1,1} & \dots & P_{1, j-1}  & P_{1,j} 
\\
& \dots  & & \dots
\\ 
P_{j-2, 1} & \dots & P_{j-2, j-1} & P_{j-2,j} 
\\
\hline
P_{j-1, 1} & \dots & P_{j-1,j-1} & P_{j-1,j} 
\end{array}\right],
\qquad 
B' =
\left|\begin{array}{ccc|c} 
1 & \dots & 1 & 1 
\\ 
P_{1,1} & \dots & P_{1, j-1}  & P_{1,k} 
\\
& \dots  & & \dots
\\ 
P_{j-2, 1} & \dots & P_{j-2, j-1} & P_{j-2,k} 
\\
\hline
P_{j-1, 1} & \dots & P_{j-1,j-1} & P_{j-1,k} 
\end{array}\right]
\\
C' &= 
\left|\begin{array}{ccc|c} 
1 & \dots & 1 & 1 
\\ 
P_{1,1} & \dots & P_{1, j-1}  & P_{1,j} 
\\
& \dots  & & \dots
\\ 
P_{j-2, 1} & \dots & P_{j-2, j-1} & P_{j-2,j} 
\\
\hline
P_{j, 1} & \dots & P_{j,j-1} & P_{j,j} 
\end{array}\right],
\qquad
D' = 
\left|\begin{array}{ccc|c} 
1 & \dots & 1 & 1 
\\ 
P_{1,1} & \dots & P_{1, j-1}  & P_{1,k} 
\\
& \dots  & & \dots
\\ 
P_{j-2, 1} & \dots & P_{j-2, j-1} & P_{j-2,k} 
\\
\hline
P_{j, 1} & \dots & P_{j,j-1} & P_{j,k} 
\end{array}\right]
\end{align*}
hence 
\begin{align*}
{\rm det}(A') = Q^{j-1}_{1, \dots, j}, \qquad
%\end{align*}
%and similarly we obtain
%\begin{align*}
{\rm det}(B') = Q^{j-1}_{1, \dots, j-1, k}.
\end{align*}
By comparing with \eqref{eq:derivata} we see that
\begin{align*}
{\rm det}(C') = \left[\sum_{i=1}^j \Delta_i \right] Q^{j-1}_{1, \dots, j};
, \qquad 
%\end{align*}
%similarly, we get that
%\begin{align*}
{\rm det}(D') = \left[\sum_{i=1}^{j-1} \Delta_i + \Delta_k \right] Q^{j-1}_{1, \dots, j-1, k}.
\end{align*}
By taking $y_1 = \dots = y_{j} = y_k = x$ and recalling \eqref{eq:derivate.uguali} we finally obtain
\begin{multline*}
Q^{j}_{1, \dots, j, k} = \frac{1}{Q^{j-2}_{1, \dots, j-1}} \left( Q^{j-1}_{1, \dots, j} \, \left[ \left( \sum_{i=1}^{j-1} \phi_i + \phi_k \right) + D_x \right] Q^{j-1}_{1, \dots, j-1, l} 
\right.
\\
\left.
- Q^{j-1}_{1, \dots, j-1, k}  \, \left[ \left( \sum_{i=1}^{j-1} \phi_i + \phi_j \right) + D_x \right] Q^{j-1}_{1, \dots, j} \right)
\end{multline*}
and simplifying the relevant terms
we obtain the thesis.
\hfill$\blacksquare$

\subsection{Proof of \cref{firstresult}}\label{sec:app-proof-first-result}

We recall from \eqref{eq:gbf-seq1} that using the Polya-Ristroph formula we have the GBF sequence
\begin{align*}
 \Psi_{j}= \dfrac{W_j}{W_{j-1}} D \left(  \dfrac{W_{j-1}}{W_j} \Psi_{j-1} \right) \quad \; \forall \; 1 \leq j \leq N,
\end{align*} 
where for simplicity we let $W_0 = 1$;
this formula can be rewritten as
\begin{align*}
\Psi_j =& \dfrac{W_j}{W_{j-1}} \left( \dfrac{W_j \, D W_{j-1}  - W_{j-1} \, D W_j  }{(W_j)^2} \Psi_{j-1} + \dfrac{W_{j-1}}{W_j}D\Psi_{j-1}  \right)
\\
=& D\Psi_{j-1} - \left(\dfrac{DW_j}{W_j} - \dfrac{DW_{j-1}}{W_{j-1}} \right)\Psi_{j-1}.
\end{align*}
Now, using \eqref{eq:determ} we have that
\begin{align*}
\dfrac{DW_j}{W_j} =
\sum_{i=1}^j P_{1,i} + \frac{D Q^{j-1}_{1, \dots, j}}{Q^{j-1}_{1, \dots, j}}, \qquad 1 \le j \le N,
\end{align*}
which implies that we can introduce the rational functions
\begin{align*}
s_1 \coloneqq P_{1,1}
\end{align*}
and, for $2 \le j \le N$,
\begin{align*}
s_j \coloneqq& \dfrac{DW_j}{W_j} - \dfrac{DW_{j-1}}{W_{j-1}} =
\frac{D Q^{j-1}_{1, \dots, j}}{Q^{j-1}_{1, \dots, j}}
- \frac{D Q^{j-2}_{1, \dots, j-1}}{Q^{j-2}_{1, \dots, j-1}}
+ P_{1,j}
= \frac{D \left( \frac{Q^{j-1}_{1, \dots, j}}{Q^{j-2}_{1, \dots, j-1}} \right)}{ \frac{Q^{j-1}_{1, \dots, j}}{Q^{j-2}_{1, \dots, j-1}} } + P_{1,j}
\end{align*}
and  $\Psi_j$ can be written as 
\begin{align*}
\Psi_j(x) = (D - s_j(x))\Psi_{j-1}(x).
\end{align*}

Recall that
\begin{align*}
\Psi_0(x) = \sum_{k=1}^N \lambda_k h_k(x),
\end{align*}
which implies that (compare with \eqref{eq:phi})
\begin{equation*}
\Psi_1(x) = (D - s_1(x)) \Psi_0(x) = \lambda_1 (D - P_{1,1})h_1 + \sum_{k=2}^N \lambda_k (D - P_{1,1}) h_k(x)
= \sum_{k=2}^N \lambda_k (P_{1,k} - P_{1,1}) h_k(x)
\end{equation*}
and recalling the definition of $Q^1_{1,k}$ in \eqref{eq:Q1} we obtain
\begin{align*}
\Psi_1(x) = \sum_{k=2}^N \lambda_k Q^1_{1,k} h_k(x)
\end{align*}
in accordance with \eqref{eq:recursive-form}.

Now we proceed by recursion.
Assume that for $2 \le j \le N$ it holds
\begin{align*}
\Psi_{j-1}(x) = \sum_{k=j}^N \lambda_k \frac{Q^{j-1}_{1, \dots, j-1, k} }{ Q^{j-2}_{1, \dots, j-1} } h_k(x).
\end{align*}
Then
\begin{align*}
\Psi_j(x) =& (D - s_j(x))\Psi_{j-1}(x) = \sum_{k=j}^N \lambda_k \left( D \frac{Q^{j-1}_{1, \dots, j-1, k} }{ Q^{j-2}_{1, \dots, j-1} } \right) h_k(x)
\\
&+ \sum_{k=j}^N \lambda_k \frac{Q^{j-1}_{1, \dots, j-1, k} }{ Q^{j-2}_{1, \dots, j-1} } D h_k(x)
- s_j \sum_{k=j}^N \lambda_k \frac{Q^{j-1}_{1, \dots, j-1, k} }{ Q^{j-2}_{1, \dots, j-1} } h_k(x).
\end{align*}
Let us denote
\begin{align*}
\Lambda^{j-1}_{k} = \frac{Q^{j-1}_{1, \dots, j-1, k} }{ Q^{j-2}_{1, \dots, j-1} }, \qquad k \ge j;
\end{align*}
then we have
\begin{align*}
\Psi_j(x) =& \sum_{k=j}^N \lambda_k \left( D \Lambda^{j-1}_{k} \right) h_k(x)
+ \sum_{k=j}^N \lambda_k \Lambda^{j-1}_{k} \, P_{1,k} \, h_k(x)
\\
&- \frac{ D \Lambda^{j-1}_{j}}{\Lambda^{j-1}_{j}} \sum_{k=j}^N \lambda_k \Lambda^{j-1}_{k} h_k(x)
- P_{1,j} \sum_{k=j}^N \lambda_k \Lambda^{j-1}_{k} h_k(x)
\\
=& \sum_{k=j+1}^N \lambda_k \left( D \Lambda^{j-1}_{k} - \frac{ D \Lambda^{j-1}_{j}}{\Lambda^{j-1}_{j}} \Lambda^{j-1}_{k}\right) h_k(x)
\\
&+ \sum_{k=j+1}^N \lambda_k \Lambda^{j-1}_{k} \, \left( P_{1,k}- P_{1,j} \right) \, h_k(x)
\\
=& \sum_{k=j+1}^N \lambda_k \, \Lambda^{j-1}_{k} \left( \frac{ D \Lambda^{j-1}_{k} }{ \Lambda^{j-1}_{k}} - \frac{ D \Lambda^{j-1}_{j}}{\Lambda^{j-1}_{j}} + \left( P_{1,k}- P_{1,j} \right) \right) h_k(x).
\end{align*}
In order to prove \eqref{eq:recursive-form} it is sufficient to prove that for any $k \ge j+1 $ it holds
\begin{align}\label{eq:to-prove}
\Lambda^{j-1}_{k} \left( \frac{ D \Lambda^{j-1}_{k} }{ \Lambda^{j-1}_{k}} - \frac{ D \Lambda^{j-1}_{j}}{\Lambda^{j-1}_{j}} + \left( P_{1,k}- P_{1,j} \right) \right)
= \frac{Q^{j}_{1, \dots, j, k} }{ Q^{j-1}_{1, \dots, j} }.
\end{align}
We will repeatedly use the identity 
\begin{align*}
\frac{Df}{f} - \frac{Dg}{g} = \frac{D(f/g)}{f/g}
\end{align*}
to compute the left hand side of the claim \eqref{eq:to-prove}: we get
\begin{align*}
\frac{Q^{j-1}_{1, \dots, j-1, k} }{ Q^{j-2}_{1, \dots, j-1} }&
\left( \frac{ D \frac{Q^{j-1}_{1, \dots, j-1, k} }{ Q^{j-2}_{1, \dots, j-1} } }{ \frac{Q^{j-1}_{1, \dots, j-1, k} }{ Q^{j-2}_{1, \dots, j-1} } } - \frac{ D \frac{Q^{j-1}_{1, \dots, j-1, j} }{ Q^{j-2}_{1, \dots, j-1} } }{ \frac{Q^{j-1}_{1, \dots, j-1, j} }{ Q^{j-2}_{1, \dots, j-1} } } + \left( P_{1,k}- P_{1,j} \right) \right)
\\
=&
\frac{Q^{j-1}_{1, \dots, j-1, k} }{ Q^{j-2}_{1, \dots, j-1} }
\left( \frac{ D \frac{Q^{j-1}_{1, \dots, j-1, k} }{ Q^{j-1}_{1, \dots, j-1, j} } }{ \frac{Q^{j-1}_{1, \dots, j-1, k} }{ Q^{j-1}_{1, \dots, j-1, j} } }  + \left( P_{1,k}- P_{1,j} \right) \right)
\\
=&
\frac{Q^{j-1}_{1, \dots, j-1, k} }{ Q^{j-2}_{1, \dots, j-1} }
\left(  \frac{D Q^{j-1}_{1, \dots, j-1, k} }{ Q^{j-1}_{1, \dots, j-1, k} } -  \frac{ DQ^{j-1}_{1, \dots, j-1, j} }{ Q^{j-1}_{1, \dots, j-1, j} } 
+ \left( P_{1,k}- P_{1,j} \right) \right)
\end{align*}
recalling \eqref{eq:Q2}, the quantity within bracket equals to 
\begin{align*}
\frac{ Q^{j}_{1, \dots, j, k} \, Q^{j-2}_{1, \dots, j-1} }{ Q^{j-1}_{1, \dots, j-1, k} \, Q^{j-1}_{1, \dots, j-1, j}  }
\end{align*}
which allows to prove \eqref{eq:to-prove}.
\hfill$\blacksquare$

\smallskip

\section{Examples}

\subsection{An example for Polynomial-Gaussian mixtures}\label{sec:ex-pgm}

In this example, we study the behaviour of the sum of two polynomial-Gaussian functions $f(x) = h_1(x) + \alpha h_2(x)$, with $\alpha \in \mathbb{R}$.
In order to provide a complete example, we fix
\begin{align*}
    h_1(x) = (x^2 + 1) e^{-x^2}, \qquad h_2(x) = (x^2 + 4x) e^{-(x-1)^2/2}
\end{align*}
so that $\mu_1 = 0$, $\sigma^2_1 = \frac12$, $\mu_2 = 1$, $\sigma^2_2 = 1$.

\begin{figure}[h!tbp]
    \centering
\begin{tikzpicture}

\draw[pattern=north west lines, pattern color=black!50] (-5, 0.90) node[anchor=east]{\footnotesize{0.8955}} rectangle  (-3.57, 3);
\draw[pattern=north east lines, pattern color=black!50] (-3.57, 0.90) rectangle  (-1.73, 2.11);
\draw[pattern=north west lines, pattern color=black!50] (-1.73, 2.11) rectangle  (-0.64, 1.67);
\draw[pattern=north east lines, pattern color=black!50] (-0.64, 1.67) rectangle  (5, 3);
\draw[pattern=north east lines, pattern color=black!50] (-5, -3) rectangle  (-3.57, 0);
\draw[pattern=north west lines, pattern color=black!50] (-0.64, -3) rectangle  (5, 0);
\draw [dashed, black!25] (-5, -3) rectangle  (5, 0);
\draw [dashed, black!25] (-5, 0.90) rectangle  (5, 1.67);
\draw [dashed, black!25] (-5, 2.11) node[anchor=east, black]{\footnotesize{2.1173}} rectangle  (5, 3);
\draw[thick, red](-5, -3) node[anchor=east, black]{\footnotesize{-3}}  -- (-5, 0);
\draw[thick, blue](-5, 0) node[anchor=east, black]{\footnotesize{0}} -- (-5, 3) node[anchor=east, black]{\footnotesize{$\alpha$}};
\draw[thick, red](5, -3) -- (5, 0);
\draw[thick, blue](5, 0) -- (5, 3);
\draw[thick, blue](-3.57, -3) -- (-3.57, 0.90);
\draw[thick, red](-3.57, 0.90) -- (-3.57, 3);
\draw[thick, blue](-1.73, -3) -- (-1.73, 2.11);
\draw[thick, red](-1.73, 2.11) -- (-1.73, 3);
\draw[thick, blue](-0.64, -3) -- (-0.64, 1.67);
\draw[thick, red](-0.64, 1.67) -- (-0.64, 3);
\node[anchor=east, black] at (-5, 1.67){\footnotesize{1.6700}};
\node[anchor= north, black] at (-5, -3){\footnotesize{$a = -5$}};
\node[anchor= north, black] at (5, -3){\footnotesize{$b = 5$}};
\node[anchor= north, black] at (-3.57, -3){\footnotesize{$x_1 = -3.57$}};
\node[anchor= north, black] at (-1.73, -3){\footnotesize{$x_2 = -1.73$}};
\node[anchor= north west, black] at (-0.74, -3){\footnotesize{$x_3 = -0.64$}};

\end{tikzpicture}
    \caption{Analysis of the grid for the function $f(x)$. The parameter $\alpha$ varies in $(-3, 3)$ and $x$ in $(-5,5)$. The orientation of the lines shows if the function is increasing or decreasing at the zero. The blue line implies that the function takes a positive value in the point of the grid, red that it is negative. No further zeros exist outside the interval $[a,b]$. Notice that for $0 < \alpha < 0.8955$ the function $f(x)$ is nonnegative on the whole real line.}
    \label{fig:1}
\end{figure}
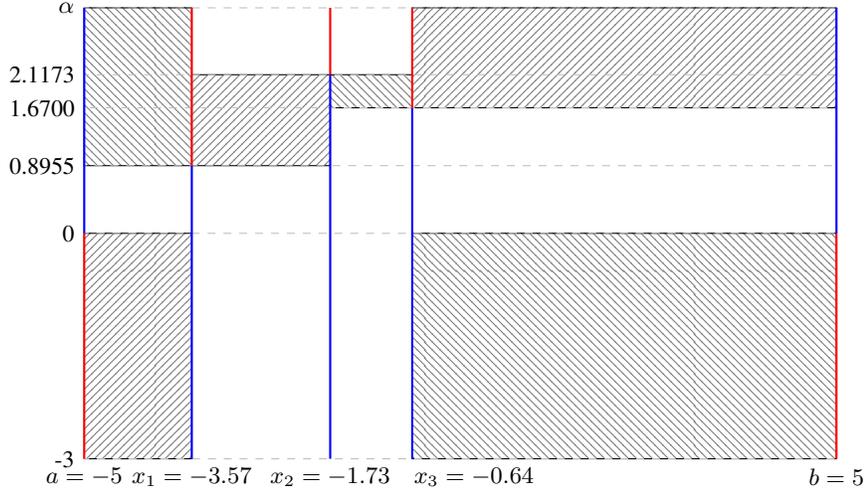

The first step 
is the choice of the interval $[a,b]$. Since the second term has larger variance, it is the dominant term for $|x| \to \infty$, hence for $\alpha > 0$ we have $f(x) > 0$ for $|x| \to \infty$, and conversely for $\alpha < 0$. Therefore, we can always choose $a \ll 0$ and $b \gg 0$ such that $f(x)$
has constant sign (equal to $\mathop{\rm sgn}(\alpha)$) outside the interval $[a,b]$.
As $\alpha$ varies in $(-10,10)$, for instance, it is sufficient to take $a \le -5$, $b \ge 5$.

Next step is the computation of the function $\psi_1(x)$. Notice that this function (and, a fortiori, the simple grid it provides for $f(x)$) is independent of the constant $\alpha$.
We obtain
\begin{align*}
    \psi_1(x) = h_{1,2}(x) = D h_2(x) - h_2(x) \frac{D h_1(x)}{h_1(x)} = \frac{1}{1+x^2} e^{-(x-1)^2/2} (4 + 6 x + x^2 + 5 x^3 + 5 x^4 + x^5)
\end{align*}
so that a simple grid for $f(x)$ is given by the real sign-changing zeros of the polynomial
\begin{align*}
    p_{1,2}(x) = 4 + 6 x + x^2 + 5 x^3 + 5 x^4 + x^5,
\end{align*}
and we obtain the following grid
\begin{align*}
    G = \{x_1 = -3.57116, x_2 = -1.72866, x_3 = -0.638509 \}.
\end{align*}
Finally, the existence of sign-changing zeros for $f(x)$ can be determined by computing the values of $f(x)$ in the points of $G \cup \{a, b\}$,
which turns out to depend linearly on $\alpha$.
Thus, we are able to construct the table in \cref{fig:1}, where the analysis of the grid (and the presence of sign-changing zeros) is shown.

According to the table, for $\alpha = 0.5$ no sign-changing zeros are present for the function $f_{0.5}(x) = h_1(x) + \frac12 h_2(x)$, while for $\alpha = 2$ there exist 4 sign-changing zeros for the function $f_{2}(x) = h_1(x) + 2 h_2(x)$. Their plots are given in \cref{fig:2}.

\begin{figure}[h!tbp]
  \centering\begin{minipage}[b]{0.45\textwidth}
    \includegraphics[width=\textwidth]{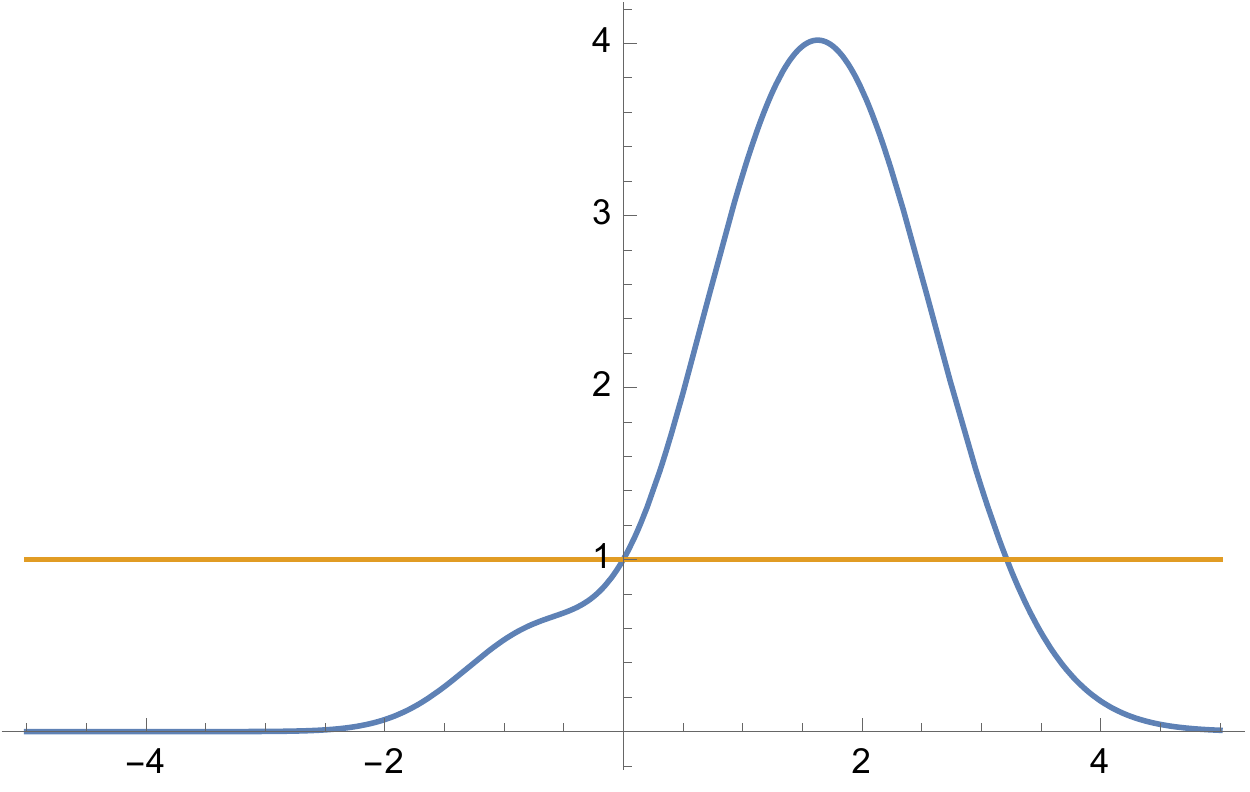}
  \end{minipage}
  \hfill
  \begin{minipage}[b]{0.45\textwidth}
    \includegraphics[width=\textwidth]{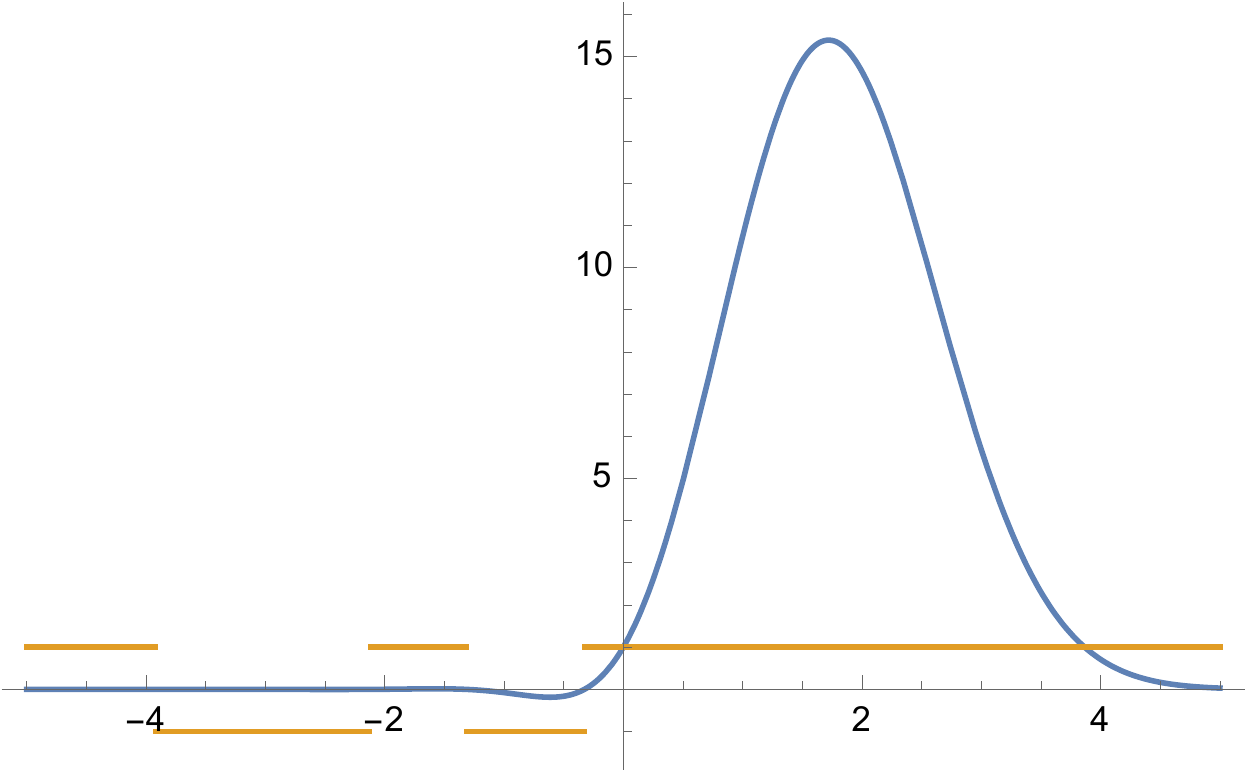}
%    \caption{Graph of the function $f_2(x)$}
  \end{minipage}
   \caption{Graph of the functions $f_{0.5}(x)$ (on the left) and $f_2(x)$ (on the right). The orange line represents the sign of the function, so to help capture the sign-changes.}
   \label{fig:2}
   \end{figure}

\subsection{An example for finite Gaussian mixtures}\label{ssec:5-example}
In this example we show how the GBF algorithm computes the sign-changing roots of a Gaussian mixture with parameters
\begin{itemize}
	\item $\mathtt{n=4};$
	\item $\mathtt{mu = [1,2,3,4]};$
	\item $\mathtt{s = [25,0.04,0.04,0.04]};$
	\item $\mathtt{lambda = [4,-1,-1,-1]}.$
	\end{itemize}
and accuracy level for Ridders' method set equal to $\mathtt{eps = 2.2204e-16}$.
The main steps executed by the software are described as follows: 
\begin{itemize}
	\item Step 1: Computation of the GBF sequence;
	\item Step 2: Computation of the bounded interval $I$ which contains all the eventual sign-changing roots.
	\item Step 3: Iterative backward scanning of the GBF sequence embedded with the computation of grids and sign-changing roots for all levels.
\end{itemize}
We mention that the implemented version of the GBF algorithm was slightly modified from its theoretical version to reduce execution times of the sign-changing roots algorithm. In fact, in Step 3 all points stemming from grids at higher levels are included in the grids for lower ones to reduce the lengths of the intervals in which Ridders' method is applied.

In the current example, the bounded interval was found to be equal to $I =[1.4921, 4.5267]$. 
The GBF functions along with their grids and sign-changing roots are displayed in \cref{fig:GBF}.

\begin{figure}[h!tbp]
	\centering
	\subfloat{
		\includegraphics[width=.45\textwidth]{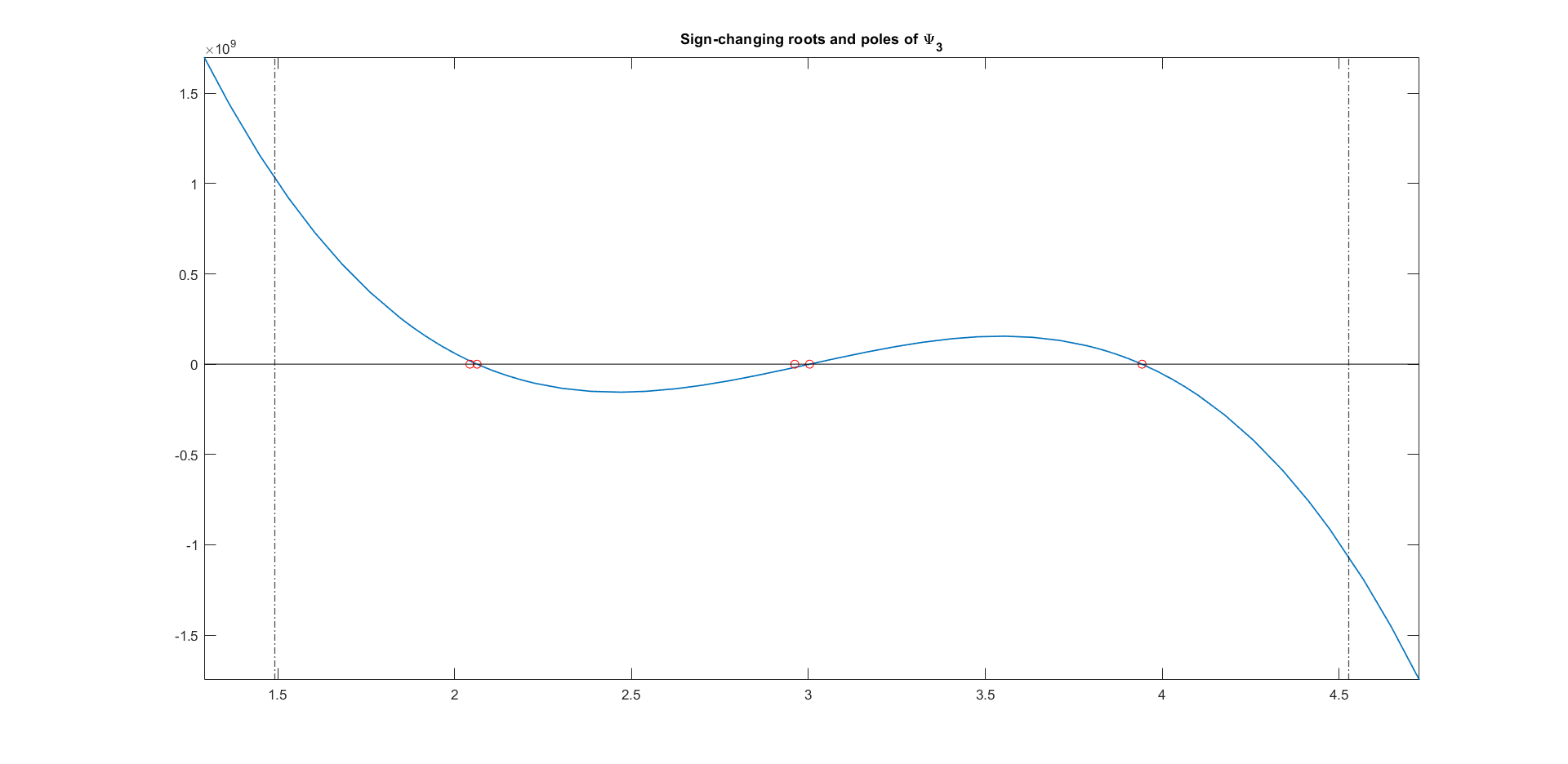}
}
	\subfloat{
		\includegraphics[width=.45\textwidth]{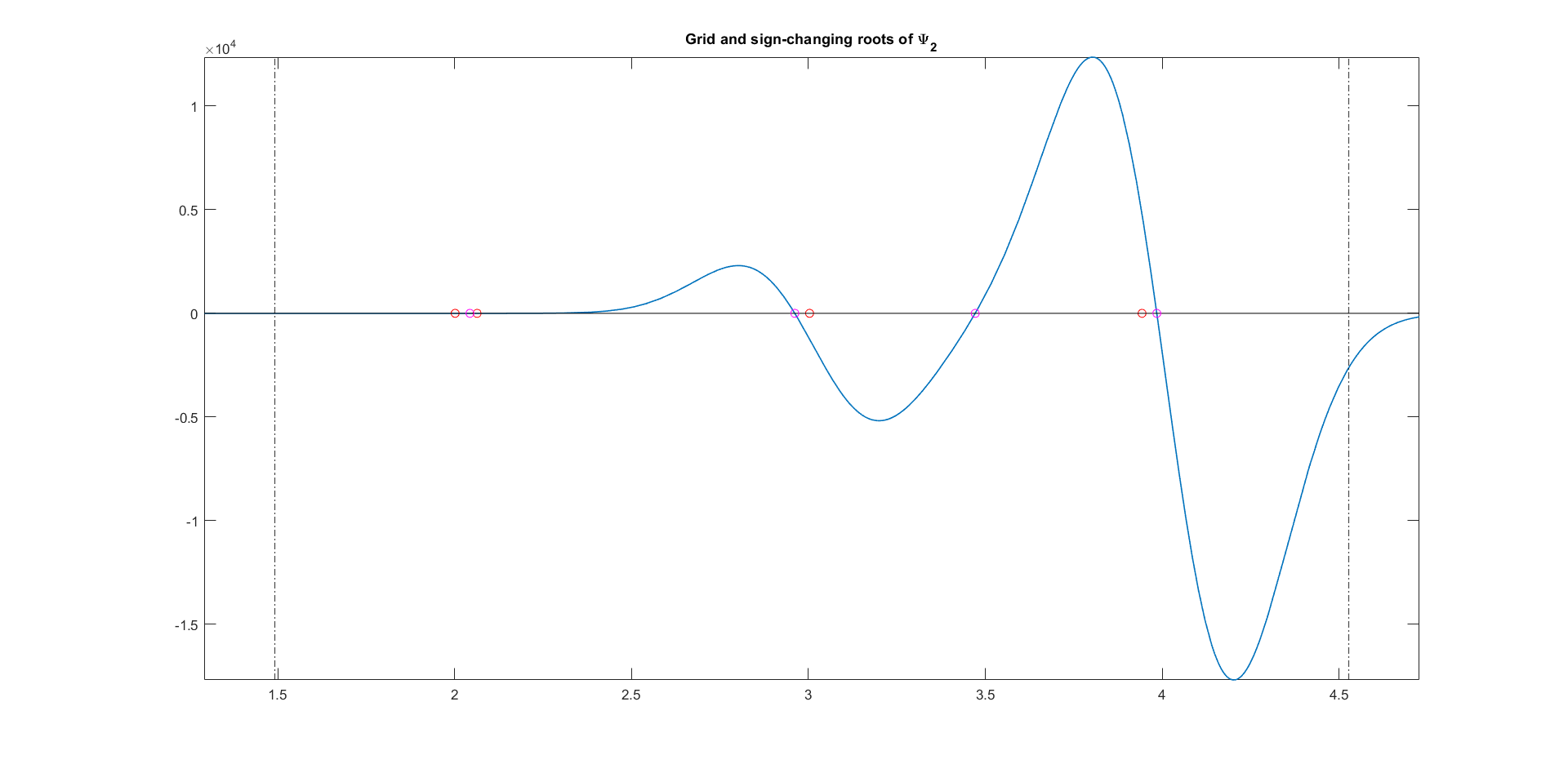}
}
\caption{GBF sequence of the Gaussian mixture discussed in  \cref{ssec:5-example}. Grids are represented in red, and sign-changing roots are represented in purple.
These plots represent the functions $\psi_3$ (on the left) and $\psi_2$ (on the right).}%
\label{fig:GBF}%
\end{figure}

\begin{figure}[h!tbp]%
%\ContinuedFloat
\centering
\subfloat{
		\includegraphics[width=.45\textwidth]{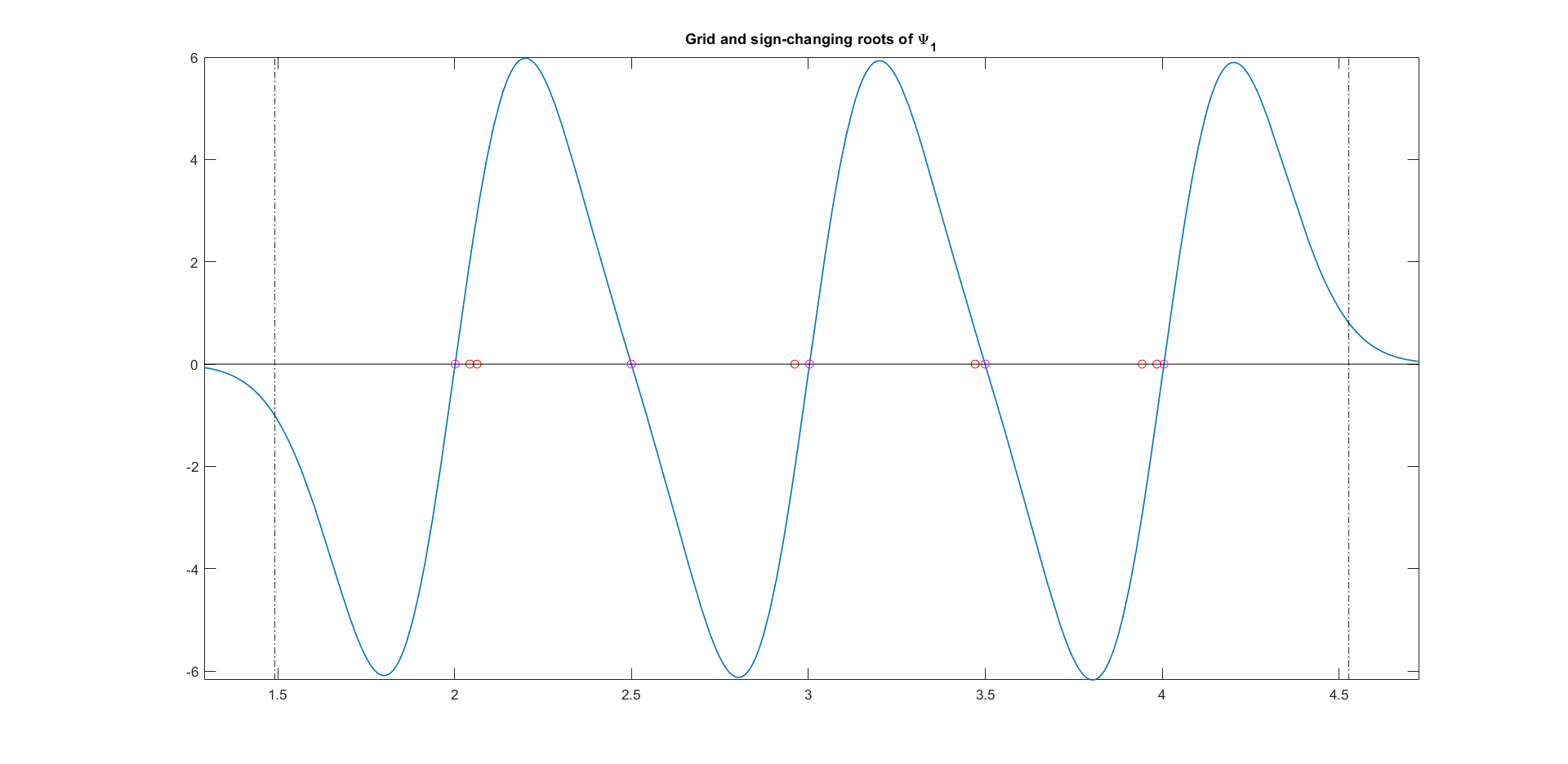}
}
	\subfloat{
		\includegraphics[width=.45\textwidth]{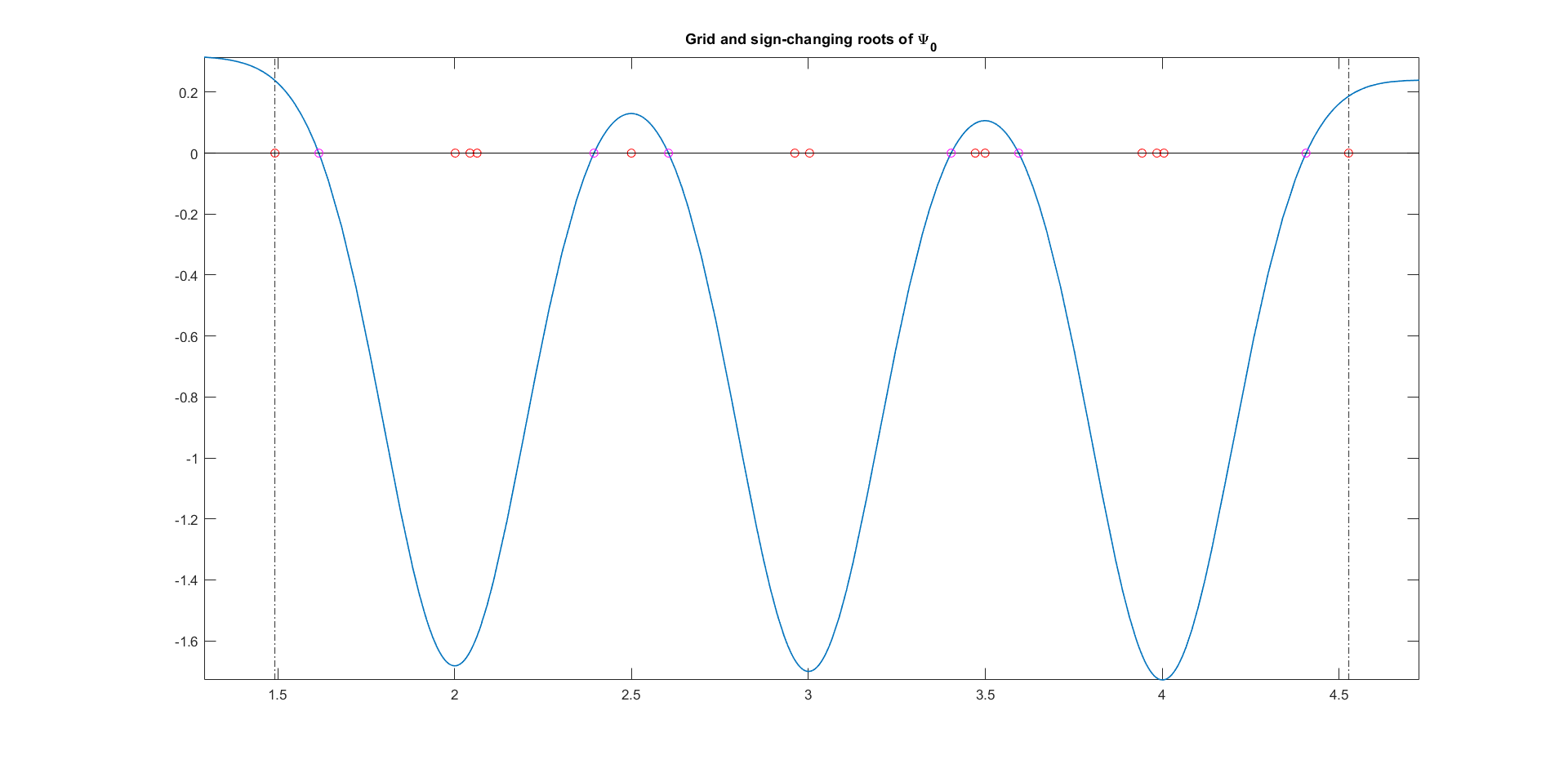}
}\caption[]{GBF sequence of the Gaussian mixture discussed in  \cref{ssec:5-example}. 
These plots represent the functions $\psi_1$ (on the left) and $f = \psi_0$ (on the right).}%
%\label{fig:GBF}%
\end{figure}	
%	\caption{GBF sequence of the Gaussian mixture discussed in  \cref{ssec:5-example}. Grids are represented in red, and sign-changing roots are represented in purple.}
%	\label{fig:GBF}
%\end{figure}

%\section*{Acknowledgments}
%This was was supported in part by......
\newpage
%Bibliography

\end{document}